\documentclass[11pt]{amsart}
\usepackage[colorlinks=true,pagebackref,hyperindex,citecolor=green,linkcolor=red]{hyperref}
\usepackage{amsmath}
\usepackage{amsfonts}
\usepackage{amssymb}
\usepackage{color}
\usepackage[all]{xy}  
\usepackage{enumerate}
\usepackage[top=1in, bottom=1in, left=0.9in, right=0.9in]{geometry}
\usepackage{mathrsfs}

\usepackage{stmaryrd}
\usepackage{verbatim}
\usepackage{cancel}
\usepackage{tikz}
\usetikzlibrary{shapes.geometric}
\usetikzlibrary{calc,shadows}
\usetikzlibrary{decorations.markings}
\theoremstyle{definition}
\newtheorem{theorem}{Theorem}[section]
\newtheorem{theoremx}{Theorem}
\numberwithin{equation}{section}

\newtheorem{question}[theorem]{Question}

\newtheorem{lemma}[theorem]{Lemma}
\newtheorem{proposition}[theorem]{Proposition}
\newtheorem{claim}[theorem]{Claim}
\newtheorem{notation}[theorem]{Notation}
\newtheorem{algorithm}[theorem]{Algorithm}

\theoremstyle{definition}
\newtheorem{definition}[theorem]{Definition}

\newtheorem{example}[theorem]{Example}

\newtheorem{remark}[theorem]{Remark}

\newtheoremstyle{TheoremNum}
        {8pt}{8pt}              %%% space between body and theorem
        {\upshape}                      %%% theorem body font
        {}                              %%% Indent amount (empty = no indent)
        {\bfseries}                     %%% theorem head font
        {.}                             %%% Punctuation after theorem head
        {.5em}                             %%% Space after theorem head
        {\thmname{#1}\thmnote{ \bfseries #3}}%%% Thm head spec
  \theoremstyle{TheoremNum}

\newcommand{\m}{\mathfrak{m}}
\newcommand{\n}{\mathfrak{n}}

\newcommand{\suchthat}{\;\ifnum\currentgrouptype=16 \middle\fi|\;}

\newcommand{\NN}{\mathbb{N}}
\newcommand{\ZZ}{\mathbb{Z}}
\newcommand{\QQ}{\mathbb{Q}}

\newcommand{\cF}{\mathcal{F}}

\newcommand{\IN}{\operatorname{in}}

\newcommand{\Spec}{\operatorname{Spec}}

\newcommand{\Ext}{\operatorname{Ext}}
\newcommand{\HH}{\operatorname{H}}

\newcommand{\Supp}{\operatorname{Supp}}
\renewcommand{\ker}{\operatorname{ker}}

\newcommand{\depth}{\operatorname{depth}}

\newcommand{\ls}{\leqslant}%
\newcommand{\gs}{\geqslant}
\newcommand{\ds}{\displaystyle}

\newcommand{\p}{\mathfrak{p}}

\newcommand{\ov}[1]{\overline{#1}}
\newcommand{\ann}{\operatorname{ann}}

\newcommand{\OO}{\mathcal{O}}

\newcommand{\PP}{\mathbb{P}}
\renewcommand{\leq}{\leqslant}
\renewcommand{\geq}{\geqslant}

\DeclareMathOperator{\coker}{coker}

\title{Decomposition of graded local cohomology tables}
\author{Alessandro De Stefani}
\address{Dipartimento di Matematica, Universit{\`a} di Genova, Via Dodecaneso 35, 16146 Genova, Italy}
\email{destefani@dima.unige.it}
\author{Ilya Smirnov}
\address{Department of Mathematics, Stockholm University, SE - 106 91 Stockholm, Sweden}
\email{smirnov@math.su.se}

\begin{document}
\begin{abstract}
Let $R$ be a polynomial ring over a field. We describe the extremal rays and the facets of the cone of local cohomology tables of finitely generated graded $R$-modules of dimension at most two. Moreover, we show that any point inside the cone can be written as a finite linear combination, with positive rational coefficients, of points belonging to the extremal rays of the cone. We also provide algorithms to obtain decompositions in terms of extremal points and facets. 
\end{abstract}
\maketitle

%%%%%%%%%%%%%%%%%%%%%%%%%%%%%%%%%%%%%%%%%%%%
\section{Introduction}
%%%%%%%%%%%%%%%%%%%%%%%%%%%%%%%%%%%%%%%%%%%%

Let $R=k[x_1,\ldots,x_m]$ be a polynomial ring over a field $k$. In 2006, Boij and S{\"o}derberg formulated two conjectures regarding the cone of Betti tables of finitely generated Cohen-Macaulay modules over $R$ \cite{BS_Conj}. 
First progress towards answering the conjectures was made by Eisenbud, Fl{\o}ystad, and Weyman \cite{EFW}, who proved the existence of modules with pure resolutions associated to any degree sequence in characteristic zero. Later on, Eisenbud and Schreyer proved the conjectures \cite{ES_JAMS}, and then Boij and S{\"o}derberg extended them to the non-Cohen-Macaulay case \cite{BS_nCM}, using the techniques introduced in \cite{ES_JAMS}.
One of the main aspects of these conjectures can roughly be summarized as follows:

\begin{theorem} Given an $R$-module $M$, there exist finitely generated graded $R$-modules $N_1,\ldots,N_s$ with a pure resolution, and positive rational numbers $r_1,\ldots,r_s$, such that
\[
\ds \beta(M) = \sum_{j=1}^s r_j \beta(N_j).
\]
Here, $\beta(-)$ denotes the Betti table of a finitely generated graded $R$-module.
\end{theorem}

At the core of the proof is the study of another object, the cone of cohomology tables of vector bundles in $\PP^{m-1}$. This cone is not dual to that of Betti tables in the usual sense. However, using suitable pairings, Eisenbud and Schreyer derive information about extremal rays and supporting hyperplanes of one cone from the other. They also provide decomposition algorithms for both cones. Later, in \cite{ES_JEMS}, the same authors extend these result to cohomology tables of coherent sheaves. 
The duality between Betti tables and cohomology tables 
was later revisited by Eisenbud and Erman in \cite{EE_JEMS}, who provided a categorified version. 
Further results on categorification for the decomposition of cohomology tables 
were proved by Erman and Sam in \cite{ES}.
Recently, there has been interest in extending the theory to other settings: 
for example, \cite{FL, FLS} develop a Boij--S\"oderberg theory for coherent sheaves on Grassmannians.

In 2015, during the Bootcamp for the AMS Summer Research Institute in Algebraic Geometry at the University of Utah, Daniel Erman asked whether a theory, analogous to that for cohomology tables of coherent sheaves, could be developed for local cohomology tables of finitely generated graded $R$-modules. In this article, we work towards answering this question. We give a complete description of the extremal rays of the cone in dimension up to two, and we show that every local cohomology table inside the cone can be expressed as a finite sum of tables from the extremal rays. In what follows, we will view lower dimensional polynomial rings as $R$-modules via the isomorphisms $k[x_1,\ldots,x_i] \cong R/(x_{i+1},\ldots,x_m)$. The following is the first main result of this article.
\begin{theoremx} (see Theorem \ref{THM decomposition}) \label{THM A}
Let $R=k[x_1,\ldots,x_m]$ be a standard graded polynomial ring, let $\m=(x_1,\ldots,x_m)R$, and $M$ be a finitely generated $\ZZ$-graded $R$-module of dimension at most two. Let $S=k[x,y]$ be a standard graded polynomial ring, and $\n=(x,y)S$. There exist positive rational numbers $r_1,\ldots,r_s$ and finitely generated graded $S$-modules $N_1,\ldots,N_s \in  \{k(a),k[x](a),S(a), \n^t(a) \mid t \in \ZZ_{\geq 1}$, $a \in \ZZ\}$ such that
\[
\ds \dim_k(\HH^i_\m(M)_\ell) = \sum_{j =1}^s r_j \dim_k(\HH^i_{\n}(N_j)_\ell)
\]
for all $i \in \{0,1,2\}$ and all $\ell \in \ZZ$. Moreover, the local cohomology tables of the modules in the set above describe the extremal rays of the cone of local cohomology tables of finitely generated graded $R$-modules of dimension at most two.
\end{theoremx}

%More generally, we study the cone of local cohomology tables of finite graded modules of dimension at most two over $k[x_1,\ldots,x_m]$. 
%Using a Noether normalization argument, one can see that the cone of local cohomology tables of finite graded modules of dimension at most two is the same as the cone of finitely generated graded $\ell[x,y]$-modules, where $\ell$ is any extension of $k$ that is infinite (see Lemma~\ref{LemmaNN}). Therefore the problem reduces to two variables (see Section~\ref{Section dim2}).

Recall that there is a well-known relation between the local cohomology of a finite module $M$ and the cohomology of the sheaf $\widetilde{M}$ associated to $M$. The relation states that  $\bigoplus_{t \in \ZZ} \HH^i(\widetilde{M}(t)) \cong \HH^{i + 1}_\m (M)$ for $i > 0$, and there is a four-terms exact sequence:
\begin{equation} \label{equation lc}
0 \to \HH^{0}_\m (M) \to M \to \bigoplus_{t \in \ZZ} \HH^0 (\widetilde{M}(t)) \to \HH^{1}_\m (M) \to 0.
\end{equation}
However, this exact sequence is a stumbling block, and we do not see a way to obtain information on the decompositions of $\HH^0_\m(M)$ and $\HH^1_\m(M)$ from those of $\bigoplus_t \HH^0(\widetilde{M}(t))$ and $M$.

To present more differences between local cohomology and sheaf cohomology, observe that, in $\PP^1$, a decomposition of cohomology tables in terms of cohomology tables of supernatural bundles is easily seen to be finite. In fact, by taking cohomology of the exact sequence $0 \to {\rm t}(\cF) \to \cF \to \cF/{\rm t}(\cF) \to 0$, where ${\rm t}(\cF)$ denotes the torsion subsheaf of the sheaf $\cF$, we obtain that it is enough to decompose the tables of ${\rm t}(\cF)$ and $\cF/{\rm t}(\cF)$ separately. For the latter, observe that $\cF/{\rm t}(\cF)$ is a direct sum of line bundles. For the former, using that $\HH^1 ({\rm t}(\cF)) = 0$ and $\dim_k \HH^0 ({\rm t}(\cF)(d)) = \chi ({\rm t}(\cF)(d))$ is a constant, one can decompose the table ${\rm t}(\cF)(d)$ using skyscraper sheaves. In the case of local cohomology tables, finiteness of the decomposition in $k[x,y]$ is a consequence of Theorem \ref{THM A}, but this requires a significant amount of work, as we will show in Section \ref{Section dim2}. In $\PP^2$, a decomposition of sheaf cohomology tables in terms of extremal points may not be finite, as shown in \cite[Example~0.3]{ES_JEMS}.  On the other hand, \cite[Theorem~0.1]{ES_JEMS} asserts that every point in the cone will be given by a convergent series of extremal points, given by supernatural cohomology tables. Given that the arguments for $\PP^1$ and $k[x,y]$ are significantly different, there is still a possibility that the decomposition of local cohomology tables in terms of extremal points is always finite. We hope to provide an answer to this question in future work.

Another important aspect of Boij--S{\" o}derberg theory is the dual description of the cone spanned by Betti tables 
by non-negative functionals. In other words, while it is very hard to say when a given table is a Betti table, 
it is possible to characterize completely tables such that some multiple is a Betti table. 
We provide an answer in the following form.

\begin{theoremx}(see Theorem~\ref{facet thm} and Algorithm~\ref{facet algo}) \label{THM B}
Let $R=k[x_1,\ldots,x_m]$, $\m=(x_1,\ldots,x_m)$, and let $\mathbb M$ denote the space of $\ZZ \times 3$ matrices with finitely many non-zero entries. 
Then (Proposition~\ref{proposition differences}) we can identify a local cohomology table of 
a finitely generated graded $R$-module of dimension at most two with a matrix in $\mathbb M$.
Furthermore,  a matrix $A = \{a_{i,j}\} \in \mathbb M$ is in 
the cone spanned by the images of local cohomology tables of $\ZZ$-graded $R$-modules of dimension at most two
if and only if the entries of $A$ satisfy the following inequalities: 
\begin{itemize}
\item $0 \leq a_{0, s}$ for $s \in \ZZ$,  
\item $0 \leq a_{1,s} + \sum_{i \leq s-1} a_{2,i}$ for $s \in \ZZ$,
\item $0 \leq a_{2,s}$ for $s \in \ZZ$,
\item $0 \leq \sum_{i > s+n} a_{1,i} + (n+1)a_{1,s+n} + \sum_{i=0}^{n-1} (i+1)a_{2,s+i}$ for $s \in \ZZ$ and $n \in \ZZ_{\geq 0}$.
\end{itemize}
\end{theoremx}

A fundamental aspect of our work are greedy decomposition algorithms accompanying both main theorems. 
In Section~\ref{Section Algorithm} and Algorithm~\ref{facet algo} we explain how to decompose, in terms of extremal points of the cone, a given local cohomology table of a finitely generated $k[x,y]$-module or a matrix satisfying the inequalities of Theorem~\ref{THM B}. We point out that the proof of Theorem~\ref{THM A} could be turned into an algorithm to obtain such a decomposition and this proof may produce a different decomposition than the one coming from the greedy algorithm of Section~\ref{Section Algorithm}. The advantage of the strategy used in the proof of Theorem \ref{THM A} is that it provides a shorter and more conceptual argument; the disadvantage is that it requires knowledge of the module $M$. The greedy algorithm provided in Section \ref{Section Algorithm}, while being less transparent and more computational in nature, only requires knowledge of the local cohomology table of $M$.

\section{Notation and background}

In what follows, let $R=k[x_1,\ldots,x_m]$ be a polynomial ring over a field $k$. We will always view $R$ with its standard grading, that is, $\deg(x_i)=1$ for all $i$. We can write $R= \bigoplus_{n \geq 0} R_n$, where $R_n$ is the $k$-vector space spanned by the monomials in $x_1,\ldots,x_d$ of degree $n$. We will use $\m$ to denote the irrelevant maximal ideal $\bigoplus_{n \geq 1} R_n$.

Local cohomology was introduced by Grothendieck \cite{Hartshorne}. One way to define it is as follows.
Given a $\ZZ$-graded $R$-module $M = \bigoplus_{n \in \ZZ} M_n$, we consider the $\check{\mbox{C}}$ech complex:
\[
\xymatrixcolsep{5mm}
\xymatrix{
\ds \check{\mbox{C}}^\bullet(M)\colon 0 \ar[r] & M \ar[r] & \ds \bigoplus_{j=1}^m M_{x_j} \ar[r] & \ds \bigoplus_{1 \leq i < j \leq m} M_{x_ix_j} \ar[r] & \ldots \ar[r] & \ds \bigoplus_{i=1}^m M_{x_1\cdots \widehat{x_i} \cdots x_m} \ar[r] & M_{x_1\cdots x_m} \ar[r] & 0,
}
\]
which is a complex of $\ZZ$-graded modules. Each map is just a localization, up to an appropriate  sign choice that makes $\check{\mbox{C}}^\bullet(M)$ into a complex. For $i \in \ZZ$, the local cohomology modules
\[
\ds \HH^i_\m(M) = \HH^i(\check{\mbox{C}}^\bullet(M))
\]
are $\ZZ$-graded Artinian $R$-modules. It is well-known that, if $\HH^i_\m(M)\ne 0$, then $\depth(M) \leq i \leq \dim(M)$ and these bounds are sharp. Given a finitely generated $\ZZ$-graded $R$-module $M$, for $j \in \ZZ$ we let
\[
\ds h^i(M)_j = \dim_k(\HH^i_\m(M)_j),
\]
where the subscript $j$ denotes the $j$-th graded component of $\HH^i_\m(M)$. It is well-known that all these dimensions are finite. We collect these numbers in a matrix with $\ZZ$-many rows, and $d+1$ columns:
\[
\ds [\HH^\bullet_\m(M)] = (h^i(M)_j)_{j \in \ZZ, 0 \leq i \leq d} \in {\rm Mat}_{\ZZ,d+1}(\ZZ_{\geq 0}).
\]
Finally, for $i=0,\ldots,d$, we denote by $[\HH^i_\m(M)]$ the $(i+1)$-st column of the matrix $[\HH^\bullet_\m(M)]$, that is, the column with entries $(h^i(M)_j)_{j \in \ZZ}$.
The following is the main question we investigate in this article.
\begin{question} \label{question}
Let $R=k[x_1,\ldots,x_m]$, where $k$ is a field. Is there a set $\Lambda_d$ of local cohomology tables of finitely generated $\ZZ$-graded $R$-modules that satisfies the following two conditions?
\begin{enumerate}
\item Given any finitely generated $\ZZ$-graded $R$-module $M$ with $\dim(M) \leq d$, there exist finitely many positive rational numbers $r_1,\ldots,r_s$ and tables $\HH_1,\ldots,\HH_s \in \Lambda_d$ such that $[\HH^\bullet_\m(M)] = \sum_{j=1}^s r_j\HH_j$.
\item The set $\Lambda_d$ is minimal, that is, none of the elements of $\Lambda_d$ can be obtain as a finite positive rational linear combination of other elements from $\Lambda_d$. 
 
\end{enumerate}
\end{question}

Observe that, if such a set $\Lambda_d$ exists, the local cohomology tables of modules from $\Lambda_d$ define the extremal rays of the cone of local cohomology tables of finitely generated graded $R$-modules. 

In relation to the above, we are also interested in a dual description of the cone, in terms of its facets. In other words, our goals include a description of the linear functionals that cut out the cone in the space of all tables.

%In other words, Question \ref{question} asks whether the local cohomology table of any finitely generated $R$-module can be constructed as a finite sum of local cohomology tables of modules from $\Lambda_d$. 

In this article we provide an answer to Question \ref{question} when $d \leq 2$ (Sections \ref{Section dim1} and \ref{Section dim2}). We first show that, in general, the study of $\Lambda_d$ reduces to understanding local cohomology tables of modules over polynomial rings in $d$ variables. Moreover, we provide the facet description of the cone, in terms of the supporting hyperplanes, again for $d\leq 2$. This is done in Section \ref{Section Algorithm facets}.

Both problems actually reduce to the study of local cohomology tables of finite graded modules over $k[x,y]$, with $k$ and infinite field, by means of the following lemma.

\begin{lemma}\label{LemmaNN}
Let $k \subseteq \ell$ be fields, with $|\ell|=\infty$, let $R=k[x_1,\ldots,x_m]$, and $S=\ell[y_1,\ldots,y_d]$. The cone of local cohomology tables of finite graded modules over $R$ of dimension at most $d$ equals the cone of local cohomology tables of finite graded $S$-modules.
\end{lemma}
\begin{proof} First observe that, when studying the cone of local cohomology tables, we may always extend the base field without losing any generality, using considerations along the lines of \cite[Lemma 9.6]{EE_JEMS}. In fact, every local cohomology table over $R$ is naturally a local cohomology table over $R_\ell = R \otimes_k \ell$; conversely, every local cohomology table over $R_\ell$ is a multiple of a local cohomology table over $R$.

We will therefore assume that $k=\ell$ is infinite, without losing any generality. Let $M$ be a finitely generated graded $R$-module of dimension at most $d$. Let $A = \ell[z_1,\ldots,z_t]$ be a graded Noether normalization of $R/\ann_R(M)$, where $t \leq d$ is forced by our assumptions. We can view $A$ as a finite graded $S$-module by sending $y_i$ to $z_i$ for $1 \leq i \leq t$, and the remaining $y_i$ to zero. Since $M$ is a finitely generated graded $R/\ann_R(M)$-module, and $S \to A \to R/\ann_R(M)$ is finite, $M$ is a finitely generated graded $S$-module with respect to the standard grading on $S$. Therefore the local cohomology table of $M$ belongs to the cone of local cohomology tables of finite $S$-modules. Conversely, every finite $S$-module can be viewed as a finite $R$-module of dimension at most $d$ via the map $R \to S$ that sends $x_i$ to $y_i$ for $1 \leq i \leq d$, and the remaining $x_i$ to zero. 
\end{proof}

\begin{remark}\label{depth and dim}
In the rest of the article, we will tacitly make use of Lemma~\ref{LemmaNN}, and 
study the cone of local cohomology tables of modules of dimension at most two
by working with polynomial rings in at most two variables over an infinite field.

Moreover, there is little harm in working with modules with positive depth. Namely,
we may decompose the table $[\HH^\bullet_\m(M)] = [\HH^0_\m(M)]+[\HH^\bullet_\m(M/\HH^0_\m(M))]$
and note that the decomposition of $\HH^0_\m(M)$ as $k$-vector space gives a decomposition of 
its local cohomology table by elements of the form $[\HH^\bullet_\m (k(a))]$.
\end{remark}

%%%%%%%%%%%%%%%%%%%%%%%%%%%%%%%%%%%%%%%%%%%%%
\section{Decomposition of graded local cohomology tables in dimension one} \label{Section dim1}
%%%%%%%%%%%%%%%%%%%%%%%%%%%%%%%%%%%%%%%%%%%%%
When $R=k$ is a field, one can immediately see that the set $\Lambda_0= \{[\HH^\bullet_\m(k(a))] \mid a \in \ZZ\}$ provides an answer to Question \ref{question}. 
Finitely generated modules over $R = k[x]$ are also very well-understood, since $R$ is a PID. We will show in this section that $\Lambda_1 = \{[\HH^\bullet_\m(k(a))], [\HH^\bullet_\m(k[x](a))] \mid a \in \ZZ\}$.

\begin{theorem} \label{THM Dim 1}
Let $R=k[x]$. The local cohomology table of every finitely generated graded $R$-module can be expressed as a finite sum, with positive integer coefficients, of local cohomology tables of the form $[\HH^\bullet_\m(k(a))]$ and $[\HH^\bullet_\m(k[x](a))]$, for $a \in \ZZ$. Moreover, the set these tables form is minimal, so that $\Lambda_1=\{[\HH^\bullet_\m(k(a))], [\HH^\bullet_\m(k[x](a))] \mid a \in \ZZ\}$ provides an answer to Question \ref{question}.
\end{theorem} 
\begin{proof}
By Remark~\ref{depth and dim}, we may assume that $M$ is positive depth and, therefore, it decomposes as a direct sum of $R(a)$.

To conclude the proof, we need to show that the set $\{[\HH^\bullet_\m(k(a))],[\HH^\bullet_\m(k[x](a))] \mid a \in \ZZ\}$ is minimal. To do so, we distinguish two cases:
\begin{enumerate}
\item First assume that there exist $\lambda_r,\mu_s \in \QQ_{\geq 0}$ such that
\[
\ds [\HH^\bullet_\m(k(a))] = \sum_{r \ne a} \lambda_r [\HH^\bullet_\m(k(r))] + \sum_{s \in \ZZ} \mu_s [\HH^\bullet_\m(k[x](s))].
\]
We will reach a contradiction by specializing these equality of $\ZZ \times 2$ tables to specific entries. In fact, the entry $(-a,1)$ on the left is $h^0(k(a))_{-a} = 1$, while every table on the right has a zero entry in that position. 
\item Now assume there exist $\lambda_r,\mu_s \in \QQ_{\geq 0}$ such that
\[
\ds [\HH^\bullet_\m(k[x](a))] = \sum_{r \in \ZZ} \lambda_r [\HH^\bullet_\m(k(r))] + \sum_{s \ne a} \mu_s [\HH^\bullet_\m(k[x](s))].
\]
Since the table on the left has all zeros in the first column, we readily get that $\lambda_r=0$ for all $r$. Moreover, since the $(-a,2)$ entry on the left is $h^1(k[x](a))_{-a} = 0$, we obtain that $\mu_s=0$ for all $s<a$. However, specializing at $(-a-1,2)$, on the left we have $h^1(k[x](a))_{-a-1} = 1$, while all the tables on the right have a zero entry in that position. A contradiction.
\end{enumerate}
\end{proof}

\section{Decomposition of graded local cohomology tables in dimension two}
\label{Section dim2}
In this section, $R$ will denote a polynomial ring $k[x,y]$ over an infinite field $k$. Given any finitely generated $R$-module $M$, we have $\HH^i_\m(M)=0$ for all $i \leq -1$ and all $i \geq 3$. Therefore the local cohomology table $\HH^i_\m(M)$ can be encoded into a $\ZZ \times 3$ matrix $[\HH^\bullet_\m(M)]$, with non-negative integer entries.  

\begin{notation}
Let $N = \bigoplus_{n\in \ZZ}N_n$ be a $\ZZ$-graded $k$-vector space that satisfies $\dim_k(N_n) < \infty$ for all $n \in \ZZ$. For $t \geq 0$ we define a ``$t$-difference function'' $\Delta^t_N\colon \ZZ \to \ZZ$ inductively. If $t=0$ then $\Delta^0_N(n) = \dim_k(N_n)$ for all $n \in \ZZ$. If $t>0$, for $n\in \ZZ$ we define $\Delta^t_N(n) = \Delta^{t-1}_N(n) - \Delta^{t-1}_N(n+1)$.
\end{notation}

\begin{proposition} \label{proposition differences} Let $R=k[x,y]$, and $M$ be a finitely generated $R$-module.
\begin{enumerate}
\item There exists an integer $a$ such that $\HH^i_\m(M)_n = 0$ for all $n > a$.
\item For $i=0,2$ we have $\Delta^j_{\HH^i_\m(M)}(n) \geq 0$ for all $n \in \ZZ$ and all $j \leq i$. For $i=1$, we have $\Delta^0_{\HH^1_\m(M)}(n) \geq 0$ for all $n \in \ZZ$, and $\Delta^1_{\HH^1_\m(M)}(n) \geq 0$ for all $n \ll 0$.
\item For every $i=0,1,2$ we have $\Delta^i_{\HH^i_\m(M)}(n)=0$ for all but finitely many $n \in \ZZ$.
\end{enumerate}
\end{proposition}
\begin{proof}
This follows from standard results on the growth of Hilbert functions of finitely generated graded modules of a given dimension. Indeed, the graded Matlis dual of $\HH^i_\m(M)$ is $\Ext^{2-i}_R(M,R(-2))$, and the latter is a finitely generated module of dimension at most $i$.
\end{proof}

In analogy with the notation we use for local cohomology modules, given a $\ZZ$-graded $R$-module $L$ we record its Hilbert function $n \mapsto \dim_k(L_n)$ in a column which we denote by $[L]$. To help keeping track of degrees, we will also include the index $n \in \ZZ$ as an extra column. Moreover, we usually represent such columns as rows, by taking the transpose matrix:
\[
\ds [L]^T = \left[ \ \ \begin{matrix} \cdots & n+1 & n & n-1 & \cdots  \\ \hline \cdots & \dim_k(L_{n+1}) & \dim_k(L_n) & \dim_k(L_{n-1}) & \cdots \end{matrix} \ \ \right].
\]
\begin{lemma} \label{decomposition_finite_length} Let $R=k[x,y]$, and $L$ be a graded cyclic $R$-module of finite length. Let $a$ (respectively, $b$) be the smallest (respectively, largest) integer $t$ such that $L_t \ne 0$. Then $[L]= \sum_{n=0}^{b-a} r_n [R/\m^{n+1}(-a)]$, for some $r_n \in \QQ_{\gs 0}$.
\end{lemma}
\begin{proof}
Since $L$ is cyclic, we can write $L=(R/I)(a)$ for some $\m$-primary homogeneous ideal $I$. Let $s=b-a$, and $d_n = \dim_k((R/I)_n)$ for all $n \in \ZZ$. By \cite[Theorem 1.1]{BoijSmith} we have that 
\begin{equation}
\label{1}
n d_n \ls (n+1) d_{n-1}
\end{equation}
for all $n \gs a$. Consider the linear system 
\begin{eqnarray*}
\left[
\begin{matrix}
0 & 0 & \ldots & 0 & s+1 \\
0 & 0 & \ldots & s & s \\
\vdots & \vdots & \vdots & \vdots \\
0 & 2 & \ldots & 2 & 2 \\
1 & 1 & \ldots &1 & 1 
\end{matrix}
\right]
\cdot 
\left[
\begin{matrix}
X_0 \\ X_1 \\ \vdots \\ X_{s-1} \\ X_s
\end{matrix}
\right] = \left[
\begin{matrix}
d_s \\ d_{s-1} \\ \vdots \\ d_1 \\ d_0
\end{matrix}
\right]
\end{eqnarray*}
which has a unique solution $(r_0,\ldots,r_s) \in \QQ^{s+1}$. We prove that $r_i  \gs 0$ for all $i$. It is clear that $r_s = d_s/(s+1) > 0$. For $0 \ls i < s$ we have that 
\[
\ds r_{i} = \frac{d_i}{i+1} - \left(\sum_{j=i+1}^s r_j\right) = \frac{d_i}{i+1} - \left(\frac{d_{i+1}}{i+2} - \frac{d_{i+2}}{i+3} \right) - \ldots - \left(\frac{d_{s-1}}{s} - \frac{d_{s}}{s+1}\right) - \frac{d_s}{s+1} = \frac{d_i}{i+1} - \frac{d_{i+1}}{i+2} \gs 0
\]
by (\ref{1}). For all $j =0,\ldots,s$, we then have that
\[
\ds d_j = \sum_{n=j}^{s} r_n(j+1) = \sum_{n=0}^{s} r_n\dim_k((R/\m^{n+1})_j). %=\sum_{n=a}^{b} r_n \dim_k((R/\m^{n+1}(-a))_j),
\]
Taking into account the shift by $a$, we finally obtain that $[L] = \sum_{n=0}^{b- a} r_n[(R/\m^{n+1})(-a)]$, as claimed.
\end{proof}

We recall the following graded versions of Serre's condition $(S_k)$
\begin{definition}
Let $(R,\m)$ be a standard graded $k$-algebra, and $M$ be a finitely generated graded $R$-module. We say that $M$ satisfies Serre's graded condition $(S_k)$ if
\[
\ds \depth(M_\p) \geq \min \{\dim(M_\p), k\}
\]
for all homogeneous ideals $\p \in \Spec(R)$. We say that $M$ satisfies Serre's graded condition $(S_k)$ on the punctured spectrum if the inequality holds for all homogeneous ideals $\p$, with $\p \ne \m$.
\end{definition}
\begin{lemma} \label{Lemma S1} Let $R=k[x_1,\ldots,x_d]$, $\m = (x_1,\ldots,x_d)$ be its irrelevant maximal ideal, and $M$ be a finitely generated $R$-module of dimension $d$. If $M$ satisfies Serre's graded condition $(S_{d-1})$ on the punctured spectrum, then $\HH^i_\m(M)$ has finite length for all $i \ne d$.
\end{lemma}
\begin{proof}
Let $i \ne d$. By graded local duality, we have that $\HH^i_\m(M)$ has finite length if and only if $\Ext^{d-i}_R(M,R)$ does. Since $M$ is graded, so is $\Ext^{d-i}_R(M,R)$. In particular, such a module has finite length if and only if $\Ext^{d-i}_R(M,R)_\p = 0$ for all homogeneous primes $\p$, with $\p \ne \m$. Given that $\Ext^{d-i}_R(M,R)_\p \cong \Ext^{d-i}_{R_\p}(M_\p,R_\p)$, by local duality the latter is zero if and only if $\HH^{i+\delta(\p)}_{\p R_\p}(M_\p) = 0$, where $\delta(\p) = \dim(R_\p)-d$. Finally, because $i \ne d$, this local cohomology module over $R_\p$ is zero given that, because of our assumptions, $M_\p$ is Cohen-Macaulay with $\dim(M_\p) = \dim(R_\p)$.
\end{proof}

The following is the main result of this section.
\begin{theorem} \label{THM decomposition}
Let $R=k[x,y]$, $\m=(x,y)$, and $M$ be a finitely generated $\ZZ$-graded $R$-module. Then $[\HH^\bullet_\m(M)]$ can be written as a finite sum with positive rational coefficients of tables of the form $[\HH^\bullet_\m(k(a))]$, $[\HH^\bullet_\m(k[x](a))]$, $[\HH^\bullet_\m(k[x,y](a))]$ and $[\HH^\bullet_\m(\m^t(a))]$, for $a \in \ZZ$. Moreover, the set of such tables is minimal. Thus, the following set provides an answer to Question \ref{question}:
\[
\ds \Lambda_2 = \{[\HH^\bullet(k(a))], [\HH^\bullet_\m(k[x](a))], [\HH^\bullet_\m(k[x,y](a))], [\HH^\bullet_\m(\m^t(a))] \mid t \in \ZZ_{\geq 1}, a\in \ZZ\}.
\]
\end{theorem}
\begin{proof}

Let $M$ be a finitely generated $R$-module, and consider its local cohomology table $[\HH^\bullet_\m(M)]$. %We denote the $j$-th entry of the $i$-th column by $h^i(M)_j$.

By Remark~\ref{depth and dim} we will assume that $M$ has positive depth. Let $\widetilde{M}$ be the sheaf on $\PP^1$ associated to $M$, so that $\widetilde{M} \cong \mathcal{F} \oplus \OO(-a_1) \oplus \cdots \oplus \OO(-a_t)$, with $a_1,\ldots,a_t \in \ZZ$ and $\mathcal{F}$ the torsion subsheaf of $\widetilde{M}$. Let $\Gamma_*(\widetilde{M}) = \bigoplus_{n \in \ZZ} \HH^0(\PP^1,\widetilde{M}(n))$, and consider the composition 
\[
\xymatrix{
M \ar[r] & \Gamma_*(\widetilde{M}) \ar[r] & \Gamma_*(\OO(-a_1) \oplus \cdots \oplus \OO(-a_t)) = R(-a_1) \oplus \cdots \oplus R(-a_t).
}
\]
We let $N$ be its kernel, and $P$ be its image. Both $N$ and $P$ have positive depth. Since $\dim(N) \leq 1$, this forces $N$ to be Cohen-Macaulay, and the exact sequence $0 \to N \to M \to P \to 0$ gives an exact sequence $0 \to \HH^1_\m(N) \to \HH^1_\m(M) \to \HH^1_\m(P) \to 0$, and $\HH^2_\m(M) \cong \HH^2_\m(P)$. Because $N$ has dimension one, it is finite over a one-dimensional polynomial ring, and it then follows from Theorem \ref{THM Dim 1} that we can decompose its table using elements from $\Lambda_2$. 
Therefore, in order to finish the proof, it suffices to show that we can decompose $[\HH^\bullet_\m(P)]$ using elements from $\Lambda_2$. We have a short exact sequence 
\[
\xymatrix{
0 \ar[r] & P \ar[r] & R(-a_1) \oplus \cdots R(-a_t) \ar[r] & C \ar[r] & 0,
}
\]
where $C$ has finite length. Taking local cohomology gives that $\HH^1_\m(P) \cong C$, and $\HH^2_\m(P) \cong \HH^2_\m(R(-a_1)) \oplus \cdots \oplus \HH^2_\m(R(-a_t))$. We induct on $t \geq 0$. If $t=0$, there is nothing to prove. If $t>0$, then we let $I(-a_t) = \ker(R(-a_t) \to C)$, which is an $\m$-primary ideal. Let $\ov{P} = \coker(I(-a_t) \to P)$ and $\ov{C} = \coker((R/I)(-a_t) \to C)$, so that we have a short exact sequence $0 \to \ov{P}\to  R(-a_1) \oplus \cdots \oplus R(-a_{t-1}) \to \ov{C} \to 0$. By induction, we can decompose $[\HH^\bullet_\m(\ov{P})]$ using tables from $\Lambda_2$. Moreover, it can be checked that $[\HH^\bullet_\m(P)] = [\HH^\bullet_\m(I(-a_t))] + [\HH^\bullet_\m(\ov{P})]$. Therefore, it suffices to decompose $[\HH^\bullet_\m(I)]$, where $I$ is an $\m$-primary ideal. By Lemma \ref{decomposition_finite_length} we can write $[R/I] = \sum_{n=0}^t r_n[R/\m^n]$ for some $r_n \in \QQ_{\geq 0}$, and some integer $t$. Observe that, since $[R/I]_0=1$, we must have $1=\sum_{n=0}^t r_n[R/\m^n]_0 = \sum_{n=0}^t r_n$. Notice that $[\HH^1_\m(I)] = [R/I] = \sum_{n=0}^t r_n [R/\m^n] = \sum_{n=0}^t r_n [\HH^1_\m(\m^n)]$. Moreover, since $\HH^2_\m(I) \cong \HH^2_\m(R) \cong \HH^2_\m(\m^n)$ for all $n$, we have that $[\HH^\bullet_\m(I)] = \sum_{n=0}^t r_n[\HH^\bullet_\m(\m^n)]$.
This concludes the proof that the local cohomology table of every module can be decomposed using tables from the set $\Lambda_2$. It is left to show the minimality of this set.

For tables of the form $[\HH^\bullet_\m(k(a))]$ and $[\HH^\bullet_\m(k[x](a))]$, the strategy is completely identical to that used inside the proof of Theorem \ref{THM Dim 1}. We therefore only focus on the proof for the remaining tables.

Assume that, for $\lambda_r,\mu_s$ and $\tau_{t,u} \in \QQ_{\geq 0}$, one has
\[
\ds [\HH^\bullet_\m(\m^n(a))] = \sum_{r \in \ZZ} \lambda_r[\HH^\bullet_\m(k(r))] + \sum_{s \in \ZZ} \mu_s [\HH^\bullet_\m(k[x](s))] + \sum_{{\tiny \begin{array}{c} u \in \ZZ \\ (t,u) \ne (n,a)\end{array}}}\tau_{t,u} [\HH^\bullet_\m(\m^t(u))].
\]
Here, we allow the exponent in $\m^t$ to be zero, in which case we mean $\m^0:=R$. Since the first column on the left contains all zeros, one readily sees that $\lambda_r=0$ for all $r$. Moreover, $\mu_s=0$ is forced for all $s$, since the table on the left satisfies $h^1(\m^n(a))_p = 0$ for $p \ll 0$. Similar considerations on zeros of the second and third column rule out $[\HH^\bullet_\m(\m^t(u))]$, with $u\ne a$. Finally, since the table on the left has zeros at $h^1(\m^n(a))_p$ for $p \geq n-a$, we have $\tau_{t,a} = 0$ for $t > n$. If $n=0$, we have reached a contradiction, since no tables on the right satisfy these requirements. If $n>0$, what is left is:
\[
\ds [\HH^\bullet_\m(\m^n(a))] = \sum_{0 \leq t < n} \tau_{t,a} [\HH^\bullet_\m(\m^t(a))].
\]
However, the entry $h^1(\m^n(a))_{n-1-a}$ on the left is equal to $n$, while on the right all the tables have zero entries. A contradiction, which concludes the proof.
\end{proof}
\begin{remark}\label{top generators}
The proof of the theorem shows that if $F$ is a graded free $R$-module such that $\HH_\m^2 (F) \cong \HH_\m^2(M)$,
then we have a surjection $F \to \HH^1_\m (C) \to 0$. 
\end{remark}

\begin{remark}
Alexandra Seceleanu has indicated to us that, quite interestingly, all modules whose local cohomology tables appear in the set $\Lambda_2$ of Theorem \ref{THM decomposition} are actually graded with respect to the fine $\ZZ^2$-grading on $R$. Daniel Erman has pointed out that they in fact satisfy an even stronger condition, as they are ${\rm GL}_2$-equivariant. Assuming Question \ref{question} has positive answer, it would be interesting to determine whether this is the case even in higher dimension.
\end{remark}

We conclude the section with an example that shows that the coefficients appearing in a decomposition may not be integers, as opposed to the case of finitely generated modules over $k[x]$. Moreover, such a decomposition may not be unique. The reason is that the cone of local cohomology tables is not simplicial, since the vectors defined by elements of $\Lambda_2$ are not linearly independent.

\begin{example} Let $M = (x^2,y^2)$. Given that $\HH^0_\m(M) = 0$, and using the isomorphisms $\HH^1_\m(M) \cong \HH^0_\m(R/(x^2,y^2))$ and $\HH^2_\m (M) \cong \HH^2_\m (R)$, one can verify that the transpose of the local cohomology table of $M$ is
\begin{align*}
[\HH^\bullet_\m(M)]^T &= 
\begin{array}{|c||c|c|c|c|c|c|c|c|c|c|}
\hline
n & \cdots  & \cdots & 2 & 1 & 0 & -1 & -2 & -3 & \cdots & \cdots \\
\hline
h^0(M)_n &  0 & \cdots &0& 0 & 0 & 0 & 0 &0& \cdots  & 0 \\
\hline
h^1(M)_n & 0 & \cdots &1& 2 & 1 & 0 & 0 & 0 & \cdots  & 0  \\
\hline
h^2(M)_n &0 & \cdots &  0 & 0 & 0 & 0 & 1 & 2  & \cdots &-n-1 \\
\hline
\end{array}\\
&= \frac{2}{3} \  \begin{array}{|c||c|c|c|c|c|c|c|c|c|}
\hline
n & \cdots & \cdots & 2 & 1 & 0 & -1 & -2 & \cdots &\cdots \\
\hline
h^0(M)_n & 0 & \cdots &0& 0 & 0 & 0 & 0 & \cdots &0 \\
\hline
h^1(M)_n & 0 & \cdots  &0& 2 & 1 & 0 & 0 & \cdots & 0 \\
\hline
h^2(M)_n &0 & \cdots & 0 & 0 & 0 & 0 & 1 & \cdots & -n-1 \\
\hline
\end{array} \\
& + 
\frac{1}{3} \ 
\begin{array}{|c||c|c|c|c|c|c|c|c|c|}
\hline
n & \cdots & \cdots & 2 & 1 & 0 & -1 & -2  & \cdots & \cdots \\
\hline
h^0(M)_n & 0&\cdots  &0& 0 & 0 & 0 & 0 &\cdots &0\\
\hline
h^1(M)_n & 0 & \cdots  &3& 2 & 1 & 0 & 0 &\cdots & 0  \\
\hline
h^2(M)_n &0 & \cdots  & 0 & 0 & 0 & 0 & 1 &\cdots & -n-1 \\
\hline
\end{array},
\end{align*}
so $[\HH^\bullet_\m(M)] = \frac{2}{3} [\HH^\bullet_\m(\m^2)] + \frac{1}{3} [\HH^\bullet_\m(\m^3)]$. Using the same module $M$, it is then easy to see that the transpose of the local cohomology table of $M \oplus R(-2)$ is
\[
[\HH^\bullet_\m(M \oplus R(-2))]^T = 
\begin{array}{|c||c|c|c|c|c|c|c|c|c|c|}
\hline
n & \cdots & \cdots &\cdots 2 & 1 & 0 & -1 & -2 & -3 & \cdots &\cdots \\
\hline
h^0_n & 0 &\cdots &0& 0 & 0 & 0 & 0 &0&\cdots &0  \\
\hline
h^1_n & 0 &\cdots  &1& 2 & 1 & 0 & 0 & 0 & \cdots & 0  \\
\hline
h^2_n &0  &\cdots & 0 & 0 & 1 & 2 & 4 & 6 & \cdots &-2n \\
\hline
\end{array},
\]
This table can then be decomposed in at least two ways: 
\[
\ds \frac{2}{3} [\HH^\bullet_\m(\m^2)] + \frac{1}{3} [\HH^\bullet_\m(\m^3)] + [\HH^\bullet_\m(R(-2))] = [\HH^\bullet_\m(M \oplus R(-2))] = [\HH^\bullet_\m(\m^2)] + [\HH^\bullet_\m(\m (-2))].
\]
\end{example}

\section{An algorithm for the decomposition of local cohomology tables in $k[x,y]$} \label{Section Algorithm}
Let $R=k[x,y]$, where $k$ is a field. We now describe a greedy algorithm that, given the local cohomology table of a finitely generated graded $R$-module, shows how to express it in terms of tables from the set $\Lambda_2$ described in Theorem \ref{THM decomposition}.

Let $L$ be a cyclic graded $R$-module of finite length. Recall that we are denoting by $[L]$ its Hilbert function, that we view as a column, where the row $n$ records the value $\dim_k(L_n)$. Let $a$ (respectively, $b$) be the smallest (respectively, largest) $n \in \ZZ$ such that $L_n \ne 0$. By Lemma \ref{decomposition_finite_length} we can write $[L] = \sum_{n=0}^{b-a} r_n[R/\m^{n+1}(-a)]$, for some $r_n \in \QQ_{\geq 0}$. We now turn the proof of Lemma \ref{decomposition_finite_length} into an explicit algorithm.

\begin{algorithm} \label{algo_finite_length}
Let $H = (h_n)$ be a $\ZZ \times 1$ matrix with non-negative rational entries. Assume that $H$ satisfies the following conditions, that we temporarily denote with $(*_a^b)$:
\begin{enumerate}
\item $h_n=0$ if and only if $n<a$ or $n>b$
\item $h_n \leq \left\{\begin{array}{ll} (n-a+1)h_a & \mbox{ if } h_{n-1}=(n-a)h_a \\ h_{n-1}h_a & \mbox{ if } h_{n-1} < (n-a)h_a\end{array} \right.$
\end{enumerate}

We describe an algorithm to write $H$ as a linear combination with non-negative rational coefficients of $[R/\m^{n+1}(-a)]$, with $0 \leq n \leq b-a$.

We proceed as follows:
\begin{itemize}
\item[\bf Step 1:] Let $r_b = \ds\frac{h_b}{b}$
\item[\bf Step 2:] Let $K=(k_n)_{n \in \ZZ}$ be the column that satisfies 
\[
\ds k_n = \left\{ \begin{array}{ll} n-a+1 & \mbox{ if }a \leq n \leq b \\
 0 & \mbox{ otherwise}\end{array} \right.
 \]
Observe that this is just $[R/\m^{b-a+1}(-a)]$. We replace $H$ by $H'=H- r_bK$.
\end{itemize}
If $H'=0$, we just write $H = r_bK = r_b[R/\m^{b-a+1}(-a)]$, and we STOP. If $H' = (h'_n)_{n \in \ZZ}$ is not the zero column, we observe that $h'_n = 0$ if and only if $n<a$ or $n>b'$, for some $0 \leq b' < b$. It takes a tedious but straightforward computation to show that $H'$ still has non-negative entries, and it satisfies $(*_a^{b'})$. We now repeat Steps 1 and 2 with $H'$, and continue until we STOP. The process clearly terminates, since every time we have a table whose number of non-zero entries decreases at least by one.
\end{algorithm}

\begin{remark}
The condition $(*_a^b)$ in Algorithm \ref{algo_finite_length} is just a restatement of Macaulay's Theorem, which characterizes the possible Hilbert functions of standard graded $k$-algebras, adapted to our setup. In particular, any cyclic $R$-module of finite length satisfies $(*_a^b)$ for some $a,b$ (see Proposition \ref{admissible HF modules}).
\end{remark}

\begin{notation} \label{notation admissible} We call a $\ZZ \times 1$ matrix $H$ that satisfies the conditions $(*_a^b)$ of Algorithm \ref{algo_finite_length} and that further satisfies $h_a=1$ and $h_n \in \NN$ for all $n \in \ZZ$ an \emph{admissible column generated in degree $a$}. Note that we do not wish to keep track of $b$ with this terminology.
If a $\ZZ \times 1$ matrix can be written as a sum of $t$ columns, each generated in degree $a_i$, we call it an \emph{admissible column, generated in degrees $a_1,\ldots,a_t$}.  Finally, given a $\ZZ \times 1$ matrix $H$, and integers $a_1,\ldots,a_t$, we set $\widetilde{H}(a_1,\ldots,a_t) = (\widetilde{h}_n)_{n \in \ZZ}$, where $\widetilde{h}_n = h_n - b_n$, and $b_n$ is the cardinality of the set $\{1,\ldots,t \mid a_i = n\}$. We call $\widetilde{H}$ the \emph{truncation of $H$ with respect to the degrees $a_1,\ldots,a_t$}. 
\end{notation}

\begin{remark} Using this new terminology, it follows from Lemma \ref{decomposition_finite_length} (or Algorithm \ref{algo_finite_length}) that every admissible column $H = (h_n)_{n \in \ZZ}$, generated in degree $a$, and such that $h_n = 0$ for $n>b$, can be realized as a sum $\sum_{n=0}^{b-a} r_n [R/\m^{n+1}(-a)]$, with $r_n \in \QQ_{\geq 0}$ and $\sum_{n=0}^{b-a}r_n = 1$.
\end{remark}

Conversely, we observe the following:

\begin{proposition} \label{admissible HF modules}
Let $L$ be a graded $R$-module of finite length, with minimal homogeneous generators of degrees $a_1,\ldots,a_t$. Then its Hilbert function $[L]$ is a finite sum of admissible columns generated in degrees $a_1, \ldots, a_t$.
\end{proposition}
\begin{proof}

Let $0 \to N \to F=R(-a_1) \oplus \cdots \oplus R(-a_t) \to L \to 0$ be a minimal free graded presentation of $L$. Choose any term order $\tau$ on $F$, and consider the initial module $\IN_\tau(N) \subseteq F$. Then $F/\IN_\tau(N)$ has the same Hilbert function as $F/N \cong L$ \cite[Theorem~15.26]{Eisenbud}. Furthermore, $\IN_\tau(N)$ consists of a direct sum of monomial ideals $I_1(-a_1) \oplus I_2(-a_2) \oplus \cdots \oplus I_t(-a_t) \subseteq F$, so that
\[
\ds F/\IN_\tau(N)  \cong (R/I_1)(-a_1) \oplus \cdots \oplus (R/I_t)(-a_t),
\]
By Macaulay's Theorem, the Hilbert function of each $R/I_j(-a_j)$ is an admissible column generated in degree $a_j$, and the proposition now follows. 
\end{proof}

We now present a series of technical lemmas regarding properties of admissible columns. These will be used in the proof of the algorithm for the decomposition. In what follows, given two columns $K=(k_n)_{n\in \ZZ}$ and $H=(h_n)_{n \in \ZZ}$, we will write $K \leq H$ if $k_n \leq h_n$ for all $n \in \ZZ$. 

\begin{lemma} \label{lemma subtraction one column}
Let $U=(u_n)_{n \in \ZZ}$ be an admissible column, generated in degree $a$, and with $u_n = 0$ for $n > b$. Let $V=(v_n)_{n \in \ZZ}$ be any column with non-negative entries such that for some integer $a' \geq a$
the following conditions hold:
\begin{enumerate}
\item $v_n = 0$ for $n < a'$ and $n > b$,
\item for all $a' \leq n \leq b$ we have $v_n \leq n-a+1$  (This condition is automatic if $V\leq L$, for some admissible column $L$ generated in degree $a$.),
\item for all $a' \leq n \leq b$ we have $v_{n}>v_{n-1}$.
\end{enumerate}
Then $W = (w_n)_{n\in \ZZ}$, defined as $w_n = \max\{0,u_n-v_n\}$, is an admissible column, and $W$ is still generated in degree $a$ if $a'>a$. Moreover, the column $Z=(z_n)_{n \in \ZZ}$ defined as $z_n = \max\{0,v_n-u_n\}$, is either zero or it satisfies $z_n > z_{n-1}$ for all $a'' \leq n \leq b$, for some $a'' \geq a'$.
\end{lemma}
\begin{proof} For the first claim, the only values we need to check for $w_n$ are those corresponding to $n$ between $a$ and $b$, since $w_n=0$ otherwise. For $a \leq n <a'$ we have $w_n = u_n$, so $w_n$ is admissible. For $a'\leq n \leq b$, if $w_n=0$ there is nothing to show. Otherwise, since $v_n > v_{n-1}$ we have $w_n = u_n - v_n \leq u_n-v_{n-1}-1$. Also, note that $u_n \leq u_{n-1}+1$ always holds. Therefore $w_n \leq u_{n-1}-v_{n-1} \leq w_{n-1}$, and thus it is admissible. If $a'>a$, then $w_n = u_n =1$, so that $W$ is generated in degree $a$.

Now, consider the column $Z$. If $Z \ne 0$, then let $n$ be an integer, with $a' \leq n \leq b$. If $u_n = n-a+1$, then since $v_n \leq n-a+1$ we must have $z_n=0$. On the other hand, if $u_{j} < j-a+1$ for some $j$, then $u_{n+1} \leq u_n$ for all $n \geq j$. If $a''$ is the smallest such value of $j$, we then have $z_{n+1} \geq v_n+1-u_n > z_n$ for all $a'' \leq n \leq b$.
\end{proof}

\begin{definition}
Given a $\ZZ \times 1$ matrix $T = (t_n)_{n \in \ZZ}$, we say that $T$ is a monotone column if $\Delta^1_{T}(n) \geq 0$ for all $n \in \ZZ$.
\end{definition} 
\begin{lemma} \label{lemma_maximal}
Let $H = (h_n)_{n \in \ZZ}$ be an admissible column generated in degrees $a_1,\ldots,a_t$. Assume that $a_1 \leq a_2 \leq \ldots \leq a_t$. Let $T = (t_n)_{n \in \ZZ}$ be a monotone column, and let $P=T+H$. Then $P$ can be written as $U + \sum_{i=1}^t K_i$, where:
\begin{itemize}
\item Each $K_i$ is an admissible column, still generated in degree $a_i$.
\item $U$ is a monotone column, with $U \leq T$.
\item $K_t$ is the maximal admissible column generated in degree $a_t$ satisfying $K_t \leq P$.
\end{itemize}
\end{lemma}
\begin{proof}
We let $K = (k_n)_{n \in \ZZ}$ be the largest admissible column generated in degree $a_t$, satisfying $K \leq P$. In other words, if $P=(p_n)_{n \in \ZZ}$, we have $k_n = \min\{p_n,n-a_t+1\}$ for all $n \geq a_t$, and $k_n=0$ otherwise.
\begin{claim} \label{claim1}
If we let $c=\min\{n \in \ZZ \mid n \geq a_t, k_n \leq k_{n-1}\}$, then $k_n=p_n$ for all $n \geq c$.
\end{claim}
\begin{proof}[Proof of the Claim] Observe that $1=k_{a_t} > k_{a_t-1} = 0$, therefore $c > a_t$. Moreover, by maximality of $K$, if $k_n > k_{n-1}$, we also have $k_{n+1} > k_n$, as long as $k_n+1 \leq p_{n+1}$. Therefore, since $k_{c-1} > k_{c-2}$ but $k_c \leq k_{c-1}$, we must have $k_{c-1}+1>p_c$. In particular, by maximality we have $k_c = p_c$. Now we recall that $H=T+H_1 + \cdots + H_t$, where each $H_i$ is admissible, generated in degree $a_i$, and $T$ is monotone. For $i=1,\ldots,t$, if we set $H_i = (h_{i,n})_{n \in \ZZ}$, we then have $p_n = t_n+\sum_i h_{i,n}$ for all $n \in \ZZ$. Observe that, for all $i$, we have $h_{i,c} \leq p_c=k_c \leq k_{c-1} \leq c-a_t \leq c-a_i < c-a_i+1$. In particular, for each $H_i$ to be admissible, we must have $h_{i,n+1} \leq h_{i,n}$ for all $n \geq c$. The same type of inequality holds for $T$, just because it is a monotone column: $t_{n+1}\leq t_n$ for all $n \in \ZZ$ and, in particular, for $n \geq c$. It follows that $p_{n+1} \leq p_n$ for all $n \geq c$, and by maximality of $K$ we then have $k_n = p_n$ for all $n \geq c$. This proves the claim.
\end{proof}
For $c$ as in Claim \ref{claim1}, and all $i=1,\ldots,t$, define $H'_{i} = (h'_{i,n})_{n \in \ZZ}$ as follows: $h'_{i,n} = h_{i,n}$ for all $n <c$, and $h'_{i,n}=0$ for all $n \geq c$. Observe that all the columns $H_i'$ are still admissible, generated in degree $a_i$. Similarly, we define $T'=(t'_n)_{n \in \ZZ}$ as follows: $t'_n=t_n$ for $n <c$, and $t_n=0$ for $n \geq c$. Observe that $T'$ is still monotone, with $T' \leq T$.

Now, we observe that $K \geq H_t$, by maximality of $K$. We define $Z_t=(z_{t,n})_{n\in \ZZ}$ as $z_{t,n} = k_n-h_{t,n}$ for $n <c$, and $z_{n,t}=0$ for $n \geq c$. By Claim \ref{claim1}, we have that $k_n>k_{n-1}$ for all $a_t \leq n < c$.  Because of this inequality, and since $K$ is admissible, we can apply Lemma \ref{lemma subtraction one column} with $U=H_t'$ and $V=K$. We then obtain that either $Z_t=0$, or $z_{t,n} > z_{t,n-1}$ for all $b_t \leq n < c$, for some $b_t>a_t$, and $z_{t,n} =0$. In case $Z_t=0$, we then have that $p_n = t_n+h_{1,n} + \cdots + h_{t-1,n} + k_n$ for all $n <c$, and $p_n = k_n$ for $n \geq c$. Thus:
\[
\ds P=T'+H_1'+ \cdots + H_{t-1}' + K
\]
is the desired decomposition, setting $U=T'$, $K_i = H_i'$ for all $i=1,\ldots,t-1$, and $K_t=K$. If $Z_t \ne 0$, observe that $z_{t,n}$ is either zero, or it satisfies $z_{t,n} \leq k_n \leq n-a_t+1$, Moreover, since $z_{t,n} > z_{t,n-1}$ for $b_t \leq n < c$, we can apply Lemma \ref{lemma subtraction one column} applied to $U=H'_{t-1}$ and $V=Z_t$. We then get that $W_{t-1}=(w_{t-1,n})_{n\in \ZZ}$, defined as $w_{t-1,n} = \max\{0,h'_{t-1,n}-z_{t,n}\}$, is admissible, generated in degree $a_{t-1}$. Moreover, $Z_{t-1}=(z_{t-1,n})_{n \in \ZZ}$, defined as $z_{t-1,n}  =\max\{0,z_{t,n}-h_{t-1,n}\}$ is either zero, or it satisfies $z_{t-1,n}>z_{t-1,n-1}$ for $b_{t-1} \leq n <c$, for some $b_{t-1} \geq b_t$. In case $Z_{t-1}=0$, we have
\[
\ds P = T'+H'_1 + \cdots + H'_{t-2} + W_{t-1} + K,
\]
using the fact that for $n <c$ one has $p_n=t_n+h_{1,n} + \cdots + h_{t-2,n} + w_{t-1,n} + k_n  = h_{t-1,n} + h_{t,n}$, while for $n \geq c$ one has $p_n = k_n$. In this case, we can set $U=T'$, $K_i = H_i'$ for $i=1,\ldots,t_2$, $K_{t-1}=W_{t-1}$, $K_t=K$ and we have the desired decomposition. If $Z_{t-1} \ne 0$, observe that $z_{t-1,n}$ is either zero, or $z_{t-1,n} \leq k_n \leq n-a_t+1$; moreover, $z_{t-1,n} > z_{t-1,n-1}$ for all $b_{t-1} \leq n < c$. We can apply again Lemma \ref{lemma subtraction one column} to $U=H'_{t-2}$ and $V=Z_{t-1}$ to obtain a column $W_{t-2}$ that is admissible, generated in degree $a_{t-2}$, and a column $Z_{t-2} = (z_{t-2,n})_{n \in \ZZ}$ defined as $z_{t-2,n} = \max\{0,z_{t-1,n}-h_{t-2,n}\}$. As before, we have that $Z_{t-2}$ is either zero, or it satisfies $z_{t-2,n}>z_{t-2,n-1}$ for all $b_{t-2} \leq n < c$, with $b_{t-2} \geq b_{t-1}$. In the first case, similar to the case above, we now have
\[
\ds P=T'+H_1' + \ldots  + H'_{t-3} + W_{t-2} + W_{t-1} + K,
\]
and we can set $U=T'$, $K_i = H_i'$ for $i=1,\ldots,t_3$, $K_{i}=W_{i}$ for $i=t-2,t-1$, and $K_t=K$. Repeating this way, we either eventually get $Z_j=0$ for some $j$, in which case
\[
\ds P=T'+H_1'+ \cdots + H_{j-1}' + W_j + \cdots + W_{t-1} + K.
\]
We can then set $U=T'$, $K_i = H_i'$ for $i=1,\ldots,j-1$, $K_{i}=W_{i}$ for $i=j,\ldots,t-1$, and $K_t=K$. Otherwise, we have constructed admissible columns $W_1,W_2,\ldots,W_{t-1}$, generated in degrees $a_1,\ldots,a_{t-1}$, and we have a column $Z_1 = (z_{1,n})_{n \in \ZZ}$ that satisfies $z_{1,n} > z_{1,n-1}$ for $b_1 \leq n < c$, and $Z_1 \leq T'$ by construction, since we started with $K \leq P$. We observe that $U=T'-Z_1$ is still monotone since $z_{1,n}>z_{1,n-1}$ for $b_1 \leq n < c$, and $t'_n=z_{1,n}=0$ for $n\geq c$. Moreover, we have $U \leq T' \leq T$. Choosing $K_i = W_i$ for all $i=1,\ldots,t-1$ and $K_t=K$, we finally have $P=U+K_1 + \cdots + K_t$, as desired.
\end{proof}

We would like to stress the fact that one should think of $K_t$ in Lemma \ref{lemma_maximal} as the ``maximal'' admissible column generated in the highest degree $a_t$, that can be subtracted from $P=T+H$.

We illustrate this construction with a concrete example.
\begin{example} Let us represent an admissible column $H = (h_n)$ generated in degree $a$ in the following way: we place a filled star in row $a$, and $h_n$-many empty circles in row $n$, with $n \ne a$. For example, the following drawing below represents the admissible column $A=(a_n)_{n \in \ZZ}$, generated in degree $-2$, with $a_{-1} = 2$, $a_{0} = 3$, $a_1 = 3$, $a_2 = 2$, $a_3=1$, and $a_n = 0$ for $n<-2$ or $n>3$:

\begin{center}
\begin{tikzpicture}
  [decoration={markings,mark=at position 1 with {\arrow{stealth}}},
   blob/.style={circle,minimum width=6pt,inner sep=0pt},
   flow/.style={postaction={decorate}}
  ]
\node () at (-1,0.5) {$\vdots$};
\node () at (-1,0) {$4$};
\node () at (-1,-0.5) {$3$};
\node () at (-1,-1) {$2$};
\node () at (-1,-1.5) {$1$};
\node () at (-1,-2) {$0$};
\node () at (-1,-2.5) {$-1$};
\node () at (-1,-3) {$-2$};
\node () at (-1,-3.5) {$\vdots$};
\node[draw,circle,inner sep=2pt] (1) at (0,-0.5) {};
\node[draw,circle,inner sep=2pt] (1) at (0,-1.5) {};
\node[draw,circle,inner sep=2pt] (1) at (0.5,-1.5) {};
\node[draw,circle,inner sep=2pt] (1) at (0,-1) {};
\node[draw,circle,inner sep=2pt] (1) at (0.5,-1) {};
\node[draw,circle,inner sep=2pt] (1) at (1,-1.5) {};
\node[draw,circle,inner sep=2pt]  at (0,-2) {};
\node[draw,circle,inner sep=2pt]  at (0.5,-2) {};
\node[draw,circle,inner sep=2pt]  at (1,-2) {};
\node[draw,circle,inner sep=2pt] (1) at (0,-2.5) {};
\node[draw,circle,inner sep=2pt] (1) at (0.5,-2.5) {};
\node[draw, fill,star, star points=5, star point ratio=.4, minimum size=2.5pt, inner sep=0pt] (1) at (0,-3) {};
\end{tikzpicture}
\end{center}
Moreover, we are going to represent a monotone column $T = (t_n)_{n \in \ZZ}$ by placing $t_n$ empty circles on line $n$. For example, the following drawing represents the monotone column that satisfies $t_n=3$ for $n \leq -1$, $t_0=2$, $t_n=1$ for $n=1,2,3$, and $t_n=0$ for $n \geq 4$: 

\begin{center}
\begin{tikzpicture}
  [decoration={markings,mark=at position 1 with {\arrow{stealth}}},
   blob/.style={circle,minimum width=6pt,inner sep=0pt},
   flow/.style={postaction={decorate}}
  ]
  \node () at (-2,-1.5) {$T=$};
\node () at (-1,0.5) {$\vdots$};
\node () at (-1,0) {$4$};
\node () at (-1,-0.5) {$3$};
\node () at (-1,-1) {$2$};
\node () at (-1,-1.5) {$1$};
\node () at (-1,-2) {$0$};
\node () at (-1,-2.5) {$-1$};
\node () at (-1,-3) {$-2$};
\node () at (-1,-3.5) {$\vdots$};
\node[draw,circle,inner sep=2pt] (1) at (0,-0.5) {};
\node[draw,circle,inner sep=2pt] (1) at (0,-1) {};
\node[draw,circle,inner sep=2pt] (1) at (0,-1.5) {};
\node[draw,circle,inner sep=2pt]  at (0,-2) {};
\node[draw,circle,inner sep=2pt]  at (0.5,-2) {};
\node[draw,circle,inner sep=2pt] (1) at (0,-2.5) {};
\node[draw,circle,inner sep=2pt] (1) at (0.5,-2.5) {};
\node[draw,circle,inner sep=2pt] (1) at (1,-2.5) {};
\node[draw,circle,inner sep=2pt] (1) at (0,-3) {};
\node[draw,circle,inner sep=2pt] (1) at (0.5,-3) {};
\node[draw,circle,inner sep=2pt] (1) at (1,-3) {};
\end{tikzpicture}
\end{center}
 
Consider the following three admissible columns, generated in degrees $-2, -2$ and $0$ respectively:

\begin{minipage}[h]{0.3\textwidth}
\begin{center}
\begin{tikzpicture}
  [decoration={markings,mark=at position 1 with {\arrow{stealth}}},
   blob/.style={circle,minimum width=6pt,inner sep=0pt},
   flow/.style={postaction={decorate}}
  ]
\node () at (-1,0.5) {$\vdots$};
\node () at (-1,0) {$4$};
\node () at (-1,-0.5) {$3$};
\node () at (-1,-1) {$2$};
\node () at (-1,-1.5) {$1$};
\node () at (-1,-2) {$0$};
\node () at (-1,-2.5) {$-1$};
\node () at (-1,-3) {$-2$};
\node () at (-1,-3.5) {$\vdots$};
\node[draw,circle,inner sep=2pt] (1) at (0,-0.5) {};
\node[draw,circle,inner sep=2pt] (1) at (0,-1.5) {};
\node[draw,circle,inner sep=2pt] (1) at (0.5,-1.5) {};
\node[draw,circle,inner sep=2pt] (1) at (0,-1) {};
\node[draw,circle,inner sep=2pt] (1) at (0.5,-1) {};
\node[draw,circle,inner sep=2pt] (1) at (1,-1.5) {};
\node[draw,circle,inner sep=2pt]  at (0,-2) {};
\node[draw,circle,inner sep=2pt]  at (0.5,-2) {};
\node[draw,circle,inner sep=2pt]  at (1,-2) {};
\node[draw,circle,inner sep=2pt] (1) at (0,-2.5) {};
\node[draw,circle,inner sep=2pt] (1) at (0.5,-2.5) {};
\node[draw, fill,star, star points=5, star point ratio=.4, minimum size=2.5pt, inner sep=0pt] (1) at (0,-3) {};
\end{tikzpicture}
\end{center}
\end{minipage}
\hspace{0.1cm}
\begin{minipage}[h]{0.3\textwidth}
\begin{center}
\begin{tikzpicture}
  [decoration={markings,mark=at position 1 with {\arrow{stealth}}},
   blob/.style={circle,minimum width=6pt,inner sep=0pt},
   flow/.style={postaction={decorate}}
  ]
\node () at (-1,0.5) {$\vdots$};
\node () at (-1,0) {$4$};
\node () at (-1,-0.5) {$3$};
\node () at (-1,-1) {$2$};
\node () at (-1,-1.5) {$1$};
\node () at (-1,-2) {$0$};
\node () at (-1,-2.5) {$-1$};
\node () at (-1,-3) {$-2$};
\node () at (-1,-3.5) {$\vdots$};
\node[draw,circle,inner sep=2pt] (1) at (0,-0.5) {};
\node[draw,circle,inner sep=2pt] (1) at (0,-1.5) {};
\node[draw,circle,inner sep=2pt] (1) at (0,-1) {};
\node[draw,circle,inner sep=2pt]  at (0,-2) {};
\node[draw,circle,inner sep=2pt]  at (0.5,-2) {};
\node[draw,circle,inner sep=2pt] (1) at (0,-2.5) {};
\node[draw,circle,inner sep=2pt] (1) at (0.5,-2.5) {};
\node[draw, fill,star, star points=5, star point ratio=.4, minimum size=2.5pt, inner sep=0pt] (1) at (0,-3) {};
\end{tikzpicture}
\end{center}
\end{minipage}
\hspace{0.1cm}
\begin{minipage}[h]{0.3\textwidth}
\begin{center}
\begin{tikzpicture}
  [decoration={markings,mark=at position 1 with {\arrow{stealth}}},
   blob/.style={circle,minimum width=6pt,inner sep=0pt},
   flow/.style={postaction={decorate}}
  ]
\node () at (-1,0.5) {$\vdots$};
\node () at (-1,0) {$4$};
\node () at (-1,-0.5) {$3$};
\node () at (-1,-1) {$2$};
\node () at (-1,-1.5) {$1$};
\node () at (-1,-2) {$0$};
\node () at (-1,-2.5) {$-1$};
\node () at (-1,-3) {$-2$};
\node () at (-1,-3.5) {$\vdots$};
\node[draw,circle,inner sep=2pt] (1) at (0,-1) {};
\node[draw, fill,star, star points=5, star point ratio=.4, minimum size=2.5pt, inner sep=0pt] (1) at (0,-1.5) {};
\end{tikzpicture}
\end{center}
\end{minipage}

Taking their sum with the monotone column $T$ defined above, we obtain

\begin{center}
\begin{tikzpicture}
  [decoration={markings,mark=at position 1 with {\arrow{stealth}}},
   blob/.style={circle,minimum width=6pt,inner sep=0pt},
   flow/.style={postaction={decorate}}
  ]
  \node () at (-3,-1.5) {$P=T+H=$};
\node () at (-1,0.5) {$\vdots$};
\node () at (-1,0) {$4$};
\node () at (-1,-0.5) {$3$};
\node () at (-1,-1) {$2$};
\node () at (-1,-1.5) {$1$};
\node () at (-1,-2) {$0$};
\node () at (-1,-2.5) {$-1$};
\node () at (-1,-3) {$-2$};
\node () at (-1,-3.5) {$\vdots$};
\node[draw,circle,inner sep=2pt] (1) at (0,-0.5) {};
\node[draw,circle,inner sep=2pt] (1) at (0.5,-0.5) {};
\node[draw,circle,inner sep=2pt] (1) at (1,-0.5) {};
\node[draw, fill,star, star points=5, star point ratio=.4, minimum size=2.5pt, inner sep=0pt] (1) at (2.5,-1.5) {};
\node[draw,circle,inner sep=2pt] (1) at (0,-1) {};
\node[draw,circle,inner sep=2pt] (1) at (0.5,-1) {};
\node[draw,circle,inner sep=2pt] (1) at (1.5,-1) {};
\node[draw,circle,inner sep=2pt] (1) at (2,-1) {};
\node[draw,circle,inner sep=2pt] (1) at (1,-1) {};
\node[draw,circle,inner sep=2pt] (1) at (0,-1.5) {};
\node[draw,circle,inner sep=2pt] (1) at (0.5,-1.5) {};
\node[draw,circle,inner sep=2pt] (1) at (1,-1.5) {};
\node[draw,circle,inner sep=2pt] (1) at (1.5,-1.5) {};
\node[draw,circle,inner sep=2pt] (1) at (2,-1.5) {};
\node[draw,circle,inner sep=2pt] (1) at (0,-2) {};
\node[draw,circle,inner sep=2pt] (1) at (0.5,-2) {};
\node[draw,circle,inner sep=2pt] (1) at (1,-2) {};
\node[draw,circle,inner sep=2pt] (1) at (1.5,-2) {};
\node[draw,circle,inner sep=2pt] (1) at (2.5,-2) {};
\node[draw,circle,inner sep=2pt] (1) at (3,-2) {};
\node[draw,circle,inner sep=2pt] (1) at (2,-2) {};
\node[draw,circle,inner sep=2pt] (1) at (0,-2.5) {};
\node[draw,circle,inner sep=2pt] (1) at (0.5,-2.5) {};
\node[draw,circle,inner sep=2pt] (1) at (1,-2.5) {};
\node[draw,circle,inner sep=2pt] (1) at (1.5,-2.5) {};
\node[draw,circle,inner sep=2pt] (1) at (2,-2.5) {};
\node[draw,circle,inner sep=2pt] (1) at (2.5,-2.5) {};
\node[draw,circle,inner sep=2pt] (1) at (3,-2.5) {};
\node[draw, fill,star, star points=5, star point ratio=.4, minimum size=2.5pt, inner sep=0pt] (1) at (1.5,-3) {};
\node[draw, fill,star, star points=5, star point ratio=.4, minimum size=2.5pt, inner sep=0pt] (1) at (2,-3) {};
\node[draw,circle,inner sep=2pt] (1) at (0,-3) {};
\node[draw,circle,inner sep=2pt] (1) at (0.5,-3) {};
\node[draw,circle,inner sep=2pt] (1) at (1,-3) {};
\end{tikzpicture}
\end{center}

We can rewrite $P$, for instance, as the sum of

\hspace{1cm}
\begin{minipage}[h]{0.3\textwidth}
\begin{center}
\begin{tikzpicture}
  [decoration={markings,mark=at position 1 with {\arrow{stealth}}},
   blob/.style={circle,minimum width=6pt,inner sep=0pt},
   flow/.style={postaction={decorate}}
  ]
  \node () at (-2,-1.5) {$U = $};
\node () at (-1,0.5) {$\vdots$};
\node () at (-1,0) {$4$};
\node () at (-1,-0.5) {$3$};
\node () at (-1,-1) {$2$};
\node () at (-1,-1.5) {$1$};
\node () at (-1,-2) {$0$};
\node () at (-1,-2.5) {$-1$};
\node () at (-1,-3) {$-2$};
\node () at (-1,-3.5) {$\vdots$};
\node[draw,circle,inner sep=2pt] (1) at (0,-1) {};
\node[draw,circle,inner sep=2pt] (1) at (0,-1.5) {};
\node[draw,circle,inner sep=2pt]  at (0,-2) {};
\node[draw,circle,inner sep=2pt]  at (0.5,-2) {};
\node[draw,circle,inner sep=2pt] (1) at (0,-2.5) {};
\node[draw,circle,inner sep=2pt] (1) at (0.5,-2.5) {};
\node[draw,circle,inner sep=2pt] (1) at (1,-2.5) {};
\node[draw,circle,inner sep=2pt] (1) at (0,-3) {};
\node[draw,circle,inner sep=2pt] (1) at (0.5,-3) {};
\node[draw,circle,inner sep=2pt] (1) at (1,-3) {};
\end{tikzpicture}
\end{center}
\end{minipage}
\hspace{1cm}
\begin{minipage}[h]{0.3\textwidth}
\begin{center}
\begin{tikzpicture}
  [decoration={markings,mark=at position 1 with {\arrow{stealth}}},
   blob/.style={circle,minimum width=6pt,inner sep=0pt},
   flow/.style={postaction={decorate}}
  ]
  \node () at (-2,-1.5) {$K_1 = $};
\node () at (-1,0.5) {$\vdots$};
\node () at (-1,0) {$4$};
\node () at (-1,-0.5) {$3$};
\node () at (-1,-1) {$2$};
\node () at (-1,-1.5) {$1$};
\node () at (-1,-2) {$0$};
\node () at (-1,-2.5) {$-1$};
\node () at (-1,-3) {$-2$};
\node () at (-1,-3.5) {$\vdots$};
\node[draw,circle,inner sep=2pt] (1) at (0,-1) {};
\node[draw,circle,inner sep=2pt] (1) at (0.5,-1) {};
\node[draw,circle,inner sep=2pt] (1) at (0,-1.5) {};
\node[draw,circle,inner sep=2pt] (1) at (1,-1.5) {};
\node[draw,circle,inner sep=2pt] (1) at (0.5,-1.5) {};
\node[draw,circle,inner sep=2pt]  at (0,-2) {};
\node[draw,circle,inner sep=2pt]  at (0.5,-2) {};
\node[draw,circle,inner sep=2pt]  at (1,-2) {};
\node[draw,circle,inner sep=2pt] (1) at (0,-2.5) {};
\node[draw,circle,inner sep=2pt] (1) at (0.5,-2.5) {};
\node[draw, fill,star, star points=5, star point ratio=.4, minimum size=2.5pt, inner sep=0pt] (1) at (0,-3) {};
\end{tikzpicture}
\end{center}
\end{minipage}

\vspace{0.5cm}

\hspace{1cm}
\begin{minipage}[h]{0.3\textwidth}
\begin{center}
\begin{tikzpicture}
  [decoration={markings,mark=at position 1 with {\arrow{stealth}}},
   blob/.style={circle,minimum width=6pt,inner sep=0pt},
   flow/.style={postaction={decorate}}
  ]
    \node () at (-2,-1.5) {$K_2 = $};
\node () at (-1,0.5) {$\vdots$};
\node () at (-1,0) {$4$};
\node () at (-1,-0.5) {$3$};
\node () at (-1,-1) {$2$};
\node () at (-1,-1.5) {$1$};
\node () at (-1,-2) {$0$};
\node () at (-1,-2.5) {$-1$};
\node () at (-1,-3) {$-2$};
\node () at (-1,-3.5) {$\vdots$};
\node[draw,circle,inner sep=2pt]  at (0,-2) {};
\node[draw,circle,inner sep=2pt]  at (0.5,-2) {};
\node[draw,circle,inner sep=2pt]  at (0,-1.5) {};
\node[draw,circle,inner sep=2pt] (1) at (0,-2.5) {};
\node[draw,circle,inner sep=2pt] (1) at (0.5,-2.5) {};
\node[draw, fill,star, star points=5, star point ratio=.4, minimum size=2.5pt, inner sep=0pt] (1) at (0,-3) {};
\end{tikzpicture}
\end{center}
\end{minipage}
\hspace{1cm}
\begin{minipage}[h]{0.3\textwidth}
\begin{center}
\begin{tikzpicture}
  [decoration={markings,mark=at position 1 with {\arrow{stealth}}},
   blob/.style={circle,minimum width=6pt,inner sep=0pt},
   flow/.style={postaction={decorate}}
  ]
    \node () at (-2,-1.5) {$K_3 = $};
\node () at (-1,0.5) {$\vdots$};
\node () at (-1,0) {$4$};
\node () at (-1,-0.5) {$3$};
\node () at (-1,-1) {$2$};
\node () at (-1,-1.5) {$1$};
\node () at (-1,-2) {$0$};
\node () at (-1,-2.5) {$-1$};
\node () at (-1,-3) {$-2$};
\node () at (-1,-3.5) {$\vdots$};
\node[draw,circle,inner sep=2pt]  at (1,-0.5) {};
\node[draw,circle,inner sep=2pt]  at (0.5,-0.5) {};
\node[draw,circle,inner sep=2pt]  at (0,-0.5) {};
\node[draw,circle,inner sep=2pt]  at (0,-1) {};
\node[draw,circle,inner sep=2pt]  at (0.5,-1) {};
\node[draw, fill,star, star points=5, star point ratio=.4, minimum size=2.5pt, inner sep=0pt] (1) at (0,-1.5) {};
\end{tikzpicture}
\end{center}
\end{minipage}

Observe that all columns $K_1,K_2$ and $K_3$ are still admissible, and they are still generated in the same degrees as the starting ones. Moreover, $K=K_3$ is the maximal admissible column generated in degree $1$ such that $K \leq P$. Additionally, $U$ is monotone, with $U \leq T$.
\end{example}

\begin{remark} \label{remark cancellation}
As a consequence of Lemma \ref{lemma_maximal}, given any admissible column $H$ generated in degrees $a_1 \leq \ldots\leq a_t$,  and any monotone column $T$, we can always construct an admissible column $K_t$, generated in the largest degree $a_t$, such that $T+H-K_t$ can be written as $U + K$, with $K$ an admissible column generated in degrees $a_1,\ldots,a_{t-1}$, and $U$ a monotone column with $U \leq T$.
\end{remark}

We observe that the same column can be admissible with respect to different degrees of generators. The following lemma allows us to extend the generating set, under certain assumptions.

\begin{lemma} \label{lemma extra generator} Let $H$ be an admissible column generated in degrees $a_1,\ldots,a_t$. Let $a\in \ZZ$, and assume that the truncation $\widetilde{H}(a_1,\ldots,a_t) = (\widetilde{h}_n)_{n \in \ZZ}$ satisfies $\widetilde{h}_a >0$. Then $H$ is an admissible column, generated in degrees $a,a_1,\ldots,a_t$.
\end{lemma}
\begin{proof}
Write $H=H_1+\cdots + H_t$, where each $H_i$ is an admissible column, generated in degree $a_i$. Since we are assuming that $\widetilde{H}(a_1,\ldots,a_t)_a > 0$, we must have $\widetilde{H_i}(a_i)_a >0$ for some $i$. Say $i=1$. We consider $K$ to be the maximal admissible column, generated in degree $a$, that satisfies $K \leq \widetilde{H_1}(a_1)$. We claim that $W = H_1 - K$ is an admissible column, generated in degree $a_1$. In fact, let $W=(w_n)_{n \in \ZZ}$, $H_1=(h_n)_{n \in \ZZ}$, and $K=(k_n)_{n \in \ZZ}$. Since $K \leq \widetilde{H_1}(a_1)$, and $K$ is generated in degree $a$, we necessarily have $a>a_1$. Moreover, we have $w_n = h_n$ for all $n < a$. In particular, $w_n=0$ for $n<a_1$ and $w_{a_1}=1$. To show that $W$ is admissible, we distinguish a few cases. For $n < a$, $w_n=h_n$, so satisfies the conditions to be admissible. For $n \geq a$, first assume that $h_{n-1}= n-a_1$, which is the maximal possible value for $H_1$ in that degree. Since $K$ is chosen to be maximal, we then must have $k_{n-1} = n-a$; observe that $k_{n-1}= n-a<n-a_1=h_{n-1}$. Moreover, we will have $h_n \leq n-a_1+1$ because $H_1$ is admissible, and $k_n = \min\{h_n,n-a+1\}$, again by maximality. In particular, we have $w_{n-1} = h_{n-1}-k_{n-1} = (n-a_1) - (n-a) = a-a_1$, and $w_n = h_n-k_n \leq (n+1-a_1) - (n+1-a) = a-a_1 = w_{n-1}$. So $W$ would be admissible in this case. On the other hand, if $h_{n-1}<n-a$, by maximality we still have $k_{n-1} = \min\{h_{n-1},n-a\}$. Thus $w_{n-1} = h_{n-1} - \min\{h_{n-1},n-a\} = \max\{0,h_{n-1}-n+a\}$. We also have $h_n \leq h_{n-1}$, because $H_1$ is admissible, and $k_n =\min\{h_{n-1},n-a+1\}$, by maximality. Therefore we get $w_n \leq \max\{0,h_{n-1}-n+a-1\} \leq \max\{0,h_{n-1}-n+a\} = w_{n-1}$. Either way, $W$ is admissible. This shows that $H = W+K + H_2+\ldots+H_t$ is admissible, generated in degrees $a,a_1,\ldots,a_t.$
\end{proof}
We are now ready to describe the algorithm. 

\begin{algorithm} \label{Algorithm} We start with the cohomology table of a finitely generated graded $R$-module $M$, that is, we start with $[\HH^\bullet_\m(M)] = (h^i(M)_j)$ for $i=0,1,2$ and $j \in \ZZ$.

We initialize $\HH = [\HH^\bullet_\m(M)]$. The goal is to describe how to subtract from $H$ positive rational combinations of elements from $\Lambda_2$ (defined as in Theorem \ref{THM decomposition}), to eventually get to the trivial table $0$. At each step, we will redefine $\HH$ to be the table we obtain from subtracting such combinations. In the end, solving for $[\HH^\bullet_\m(M)]$ will result in the desired decomposition of $[\HH^\bullet_\m(M)]$. Throughout, we denote with $\HH^0, \HH^1$ and $\HH^2$ the first, second, and third column of $\HH$, respectively. Moreover, we denote by $h^i_j$ the entry in row $j \in \ZZ$ of the column $\HH^i$. 

\begin{itemize}
\item[\bf Step 1] Replace $\HH$ by $\HH-\sum_{n \in \ZZ} h^0_n [\HH^\bullet_\m(k)(-n)]$.  \\

\item[\bf Step 2] If the set $\{ n \in \ZZ \mid  \Delta^2_{\HH^2}(n-2) \ne 0\}$ is empty, go to {\bf Step 4}. Otherwise, let $a$ be its maximum. If $h^1_a = 0$, replace $\HH$ by $\HH-[\HH^\bullet_\m(R)(-a)]$. If $h^1_a >0$, proceed to {\bf Step 3}.

\item[\bf Step 3] 
Set $K = (k_n)_{n \in \ZZ}$ by
\[
\ds k_n = \left\{ \begin{array}{ll} \min\{h^1_n,n-a+1\} & \mbox{ if } n\geq a \\ 0 & \mbox{ if } n<a \end{array} \right.
\]
Use Algorithm \ref{algo_finite_length} to write $K = \sum_n r_n[R/\m^{n+1}(-a)]$ for some $r_n \in \QQ_{\geq 0}$. Replace $\HH$ by $\HH-\sum_nr_n [\HH^\bullet_\m(\m^{n+1})(-a)]$ and return to {\bf Step 2}. \\

\item[\bf Step 4] If the set $\{n \in \ZZ \mid  \Delta^1_{\HH^1}(n-1)  \ne  0\}$ is empty, then FINISH. Otherwise, let $b$ be its maximum. Replace $\HH$ by $\HH-[\HH^\bullet_\m(k[x])(-b)]$ and repeat {\bf Step 4}. \\
\end{itemize}

\end{algorithm}

\begin{proof}
We prove that the Algorithm terminates with the trivial table $\HH = 0$, and thus produces the desired decomposition of $[\HH^\bullet_\m(M)]$. 
{\bf Step 1} removes the first column, that is, the one corresponding to $\HH^0_\m(M)$. Note that we are subtracting only a finite sum of tables of the form $[\HH^\bullet_\m(k)(-n)]$, because of Proposition \ref{proposition differences}. 

By collecting the values of $a$ from {\bf Step 2} that correspond to $h^1_a>0$, we obtain a sequence of integers $a_1 \geq a_2 \geq \cdots \geq a_t$ that satisfies the following three conditions:
\begin{enumerate}
\item $\HH^2_\m(M) \cong \HH^2_\m(R)(-a_1) \oplus \cdots \oplus \HH^2_\m(R)(-a_t) \oplus F$,
where $F$ is a graded free module generated in degrees corresponding to $a$ with $h^1_a = 0$.  
\item If $0 \to T \to M/\HH^0_\m(M) \to C \to 0$ is a short exact sequence as in the proof of Theorem \ref{THM decomposition}, then there is a surjection $R(-a_1) \oplus \cdots \oplus R(-a_t) \to \HH^1_\m(C) \to 0$.
\item For all $i=0,\ldots,t-1$, if we let $\widetilde{H^1}(a_1,\ldots,a_i) = (\widetilde{h}_n)_{n \in \ZZ}$, we have $\widetilde{h}_{a_{i+1}} > 0$. 
\end{enumerate}
The first two claims follow from Remark~\ref{top generators}. The third condition comes from the way the sequence $a_1,\ldots,a_t$ appears in {\bf Step 2}.

Now, recall that in the proof of Theorem \ref{THM decomposition} it is shown that $[\HH^1(M)] = [\HH^1_\m(T)] + [\HH^1_\m(C)]$, where this decomposition comes from the condition (2) described above. Since $\HH^1_\m(C)$ has finite length with generators of degrees contained in the set $\{a_1,\ldots,a_t\}$, by Proposition \ref{admissible HF modules}, its Hilbert function $[\HH^1_\m(C)]$ is an admissible column, generated in degrees contained in the set $\{a_1,\ldots,a_t\}$. Because of  condition (3) above, we may use Lemma \ref{lemma extra generator} to extend the generating set and 
assume that $[\HH^1_\m(C)]$ is an admissible column generated in all degrees $a_1,\ldots,a_t$. Moreover, $[\HH^1_\m(T)]$ is monotone, by Proposition~\ref{proposition differences}. Therefore $\HH^1$ is the sum of a monotone column, and an admissible column, with generators in degrees $a_1 \geq \ldots \geq a_t$.

At each iteration of {\bf Step 3} the constructed column $K$  is, by definition, 
the maximal admissible column generated in the largest possible degree $a$ and such that $K \leq \HH^1$.
This column is decomposed using using Algorithm \ref{algo_finite_length} as a
non-negative rational linear combination of the tables of $[R/\m^{n+1}(-j)]$. Recall that the table $[\HH^\bullet_\m(\m^{n+1})(-a)]$ has second column equal to $[R/\m^{n+1}(-a)]$. Moreover, we have $\HH^2_\m(\m^{n+1})(-a) \cong \HH^2_\m(R)(-a)$ for all $n$. Since, as shown in the proof of Theorem \ref{THM decomposition}, we have $\sum_n r_n=1$, we conclude that $\sum_n r_n [\HH^\bullet_\m(\m^{n+1}(-a)]$ has:
\begin{enumerate}
\item First column equal to zero.
\item Second column equal to $K$. In particular, by Remark \ref{remark cancellation}, the second column of the table $\HH - \sum_n r_n [\HH^\bullet_\m(\m^{n+1}(-a)]$ is equal to some $U + A$, where $U$ is monotone with $U \leq  [\HH^1_\m(T)]$, and $A$ is still admissible, now generated in the remaining degrees $a_i, \ldots,a_{t}$. 
\item Third column equal to the third column of $[\HH^\bullet_\m(R)(-a)]$. In particular, by condition (3) above, 
the third column of $\HH-\sum_n r_n [\HH^\bullet_\m(\m^{n+1}(-a)]$ is equal to $[\HH^2_\m(G)]$
where $G$ is a free module generated in the remaining degrees $\leq a$. 
\end{enumerate}  

Thus after repeating {\bf Step 2} and {\bf Step 3} as required, we eliminate the third column of $\HH$. Moreover, the remaining second column, $\HH^1$, is now monotone with $\HH^1 \leq [\HH^1_\m(T)]$.

{\bf Step 4} constructs the monotone column $U$ remaining in $\HH^1$ using tables of the form $[\HH^\bullet_\m(k[x](-j))]$. By Proposition \ref{proposition differences}, $\{n \in \ZZ \mid \Delta^1_{\HH^1_\m(T)}(n) \ne 0\}$ is a finite set and $\Delta^1_{\HH^1_\m(T)}(n)<\infty$ for all $n$.
Since $U$ is monotone and $U \leq [\HH^1_\m(T)]$ by Lemma \ref{lemma_maximal}, it follows that $0 \leq \Delta^1_U (n) < \infty$ for all $n$, and it is zero for all but finitely many values of $n$. Note that $\Delta^1_{\HH^1_\m(k[x](-j))}(n) = 1$ if $n = -j$ and is $0$ for all other values of $n$. Thus, each iteration of {\bf Step 4} decreases by $1$ precisely one nonzero entry of $\Delta^1_U$, and the algorithm returns the zero table after finitely many steps.
\end{proof}

\begin{remark} Let $M$ be a finitely generated $k[x]$-module. If one runs Algorithm \ref{Algorithm} with the table $[\HH^\bullet_\m(M)]$, ignoring Step 2 and Step 3 (which require a third column in the matrix), then one gets a decomposition of $[\HH^\bullet_\m(M)]$ in terms of the extremal points $\{[\HH^\bullet_\m(N)] \mid N \in \Lambda_1\}$ of the cone in dimension one, as described in Theorem \ref{THM Dim 1}.
\end{remark}

\begin{example}
Consider the following $R=k[x,y]$-module:
\[
M = {\rm coker}
\begin{pmatrix} 
x^3 & x^2y^2 & x^4y^2\\
x^2y & x^3y + xy^3 & x^4y^2 + x^2y^4\\
x^3 + y^3 & x^4 & x^3y^3\\
x^3 & 2x^2y^2 & x^5y\\
y^3 & y^4 & x^6
\end{pmatrix}
\]

Using Macaulay 2 \cite{M2}, one can check that $M$ has transposed local cohomology table 
\[
[\HH^\bullet_\m(M)]^T = 
\begin{array}{|c||c|c|c|c|c|c|c|c|c|c|c|c|c|c|}
\hline
n & 5 & 4 & 3 & 2 & 1 & 0 & -1 & -2 & -3 & -4 & -5 & -6 & -7 & -8 \\
\hline
h^0(M)_n &0& 0 & 0 & 0 & 0 &0&0 & 0& 0&  0 & 0 & 0 & 0 & 0 \\
\hline
h^1(M)_n &0& 1 & 2 & 4 & 7 &10&13 & 11 & 9 &  7 &  5 & 3  & 2  & 2 \\
\hline
h^2(M)_n & 0 & 0 & 0 & 0 & 0 & 0 & 0  & 0   & 0 &  0 & 0  & 0  &  1 & 3\\
\hline
\end{array}
\]

From now on, since the column $[\HH^0_\m(M)]$ consists of all zeros, we will disregard it. 

The first meaningful step in the algorithm is Step 3: $a = -5$ gives an admissible column $K = (k_n)_{n \in \ZZ}$, generated in degree $-5$, that we can write as
\begin{align*}
\begin{bmatrix}
n & k_n\\
4 & 1\\
3 & 2 \\
2 & 4 \\
1 & 7 \\
0 & 6 \\
-1 & 5\\
-2 & 4\\
-3 & 3\\
-4 & 2\\
-5 & 1
\end{bmatrix}
 & = \frac 1{10} 
\begin{bmatrix}
n & \\
4 & 10\\
3 & 9 \\
2 & 8 \\
1 & 7 \\
0 & 6 \\
-1 & 5\\
-2 & 4\\
-3 & 3\\
-4 & 2\\
-5 & 1
\end{bmatrix}
+ \frac {11}{90}
\begin{bmatrix}
n & \\
4 & 0\\
3 & 9 \\
2 & 8 \\
1 & 7 \\
0 & 6 \\
-1 & 5\\
-2 & 4\\
-3 & 3\\
-4 & 2\\
-5 & 1
\end{bmatrix}
+ \frac {5}{18}
\begin{bmatrix}
n & \\
4 & 0\\
3 & 0 \\
2 & 8 \\
1 & 7 \\
0 & 6 \\
-1 & 5\\
-2 & 4\\
-3 & 3\\
-4 & 2\\
-5 & 1
\end{bmatrix}
+ \frac 12 
\begin{bmatrix}
n & \\
4 & 0\\
3 & 0 \\
2 & 0 \\
1 & 7 \\
0 & 6 \\
-1 & 5\\
-2 & 4\\
-3 & 3\\
-4 & 2\\
-5 & 1
\end{bmatrix}
\\ \\ 
 & = \frac{1}{10} \ [\HH^1_\m(\m^{11}(5))] + \frac{11}{90} \ [\HH^1_\m(\m^{10}(5))] + \frac{5}{18} \ [\HH^1_\m(\m^9(5))] + \frac{1}{2} \ [\HH^1_\m(\m^8(5))].
\end{align*}
Subtracting $\frac{1}{10} \ [\HH^\bullet_\m(\m^{11}(5))] + \frac{11}{90} \ [\HH^\bullet_\m(\m^{10}(5))] + \frac{5}{18} \ [\HH^\bullet_\m(\m^9(5))] + \frac{1}{2} \ [\HH^\bullet_\m(\m^8(5))]$ from $[\HH^\bullet_\m(M)]$, we get
\[
\begin{array}{|c||c|c|c|c|c|c|c|c|c|c|c|c|c|c|}
\hline
n           & 5 & 4 & 3 & 2 & 1 & 0 & -1 & -2 & -3 & -4 & -5 & -6 & -7 & -8 \\
\hline
h^1(M)_n &0& 0 & 0 & 0 & 0 &4&8 & 7 & 6 &  5 &  4 & 3  & 2  & 2 \\
\hline
h^2(M)_n& 0 & 0 & 0 & 0 & 0 & 0 & 0  & 0   & 0 &  0 & 0  & 0  &  0 & 1\\
\hline
\end{array}
\]
Next, for $a = -6$ we construct an admissible column $K = (k_n)_{n\in \ZZ}$, generated in degree $-6$, as follows:
\begin{align*}
& \begin{bmatrix}
n & k_n\\
0 & 4\\
-1 & 6\\
-2 & 5\\
-3 & 4\\
-4 & 3\\
-5 & 2 \\
-6 & 1 \\
\end{bmatrix}
= \frac 47 
\begin{bmatrix}
n & \\
0 & 7\\
-1 & 6\\
-2 & 5\\
-3 & 4\\
-4 & 3\\
-5 & 2 \\
-6 & 1 \\
\end{bmatrix}
+ \frac 37
\begin{bmatrix}
n & \\
0 & 0\\
-1 & 6\\
-2 & 5\\
-3 & 4\\
-4 & 3\\
-5 & 2 \\
-6 & 1 \\
\end{bmatrix}
\\ \\
& = \frac{4}{7} \ [\HH^1_\m(\m^8(6))] + \frac{3}{7} [\HH^1_\m(\m^7(6))].
\end{align*}
Subtracting $\frac{4}{7} \ [\HH^\bullet_\m(\m^8(6))] + \frac{3}{7} [\HH^\bullet_\m(\m^7(6))]$ leaves the table 
\begin{align*}
& \begin{array}{|c||c|c|c|c|c|c|c|c|c|c|c|c|c|c|}
\hline
n           & 5 & 4 & 3 & 2 & 1 & 0 & -1 & -2 & -3 & -4 & -5 & -6 & -7 & -8 \\
\hline
h^1(M)_n &0& 0 & 0 & 0 & 0 & 0&2 & 2 & 2  &  2 &  2 & 2  & 2  & 2 \\
\hline
h^2(M)_n& 0 & 0 & 0 & 0 & 0 & 0 & 0  & 0   & 0 &  0 & 0  & 0  &  0 & 0\\
\hline
\end{array} \\ \\
& = 2[\HH^\bullet(k[x])].
\end{align*}

\end{example}

\section{Facets of the cone of local cohomology tables in dimension two} \label{Section Algorithm facets}
We adopt the following notation. 
In the space of $3 \times \mathbb Z$-matrices let $\mathbb M$ denote the 
subspace formed by the matrices with finitely many nonzero entries. 
We consider the cone $C \subseteq \mathbb M$ generated by the matrices
$\Delta\Lambda_2 = \{E_{i,s}, \Gamma_s(n) \mid i \in \{0, 1, 2\}, s \in \ZZ, n \in \ZZ_{\geq 1} \}$, where 
$E_{i,s}$ are the elementary matrices $\{e_{i,s} = 1, e_{j, t} = 0 \text { for } (i, s) \neq (j, t) \}$ and
\[
\Gamma_s(n) = 
\begin{array}{|c||c|c|c|c|c|c|c|}
\hline
i & \cdots & s+n+1 & s+n & \cdots & s+1 & s & \cdots\\
\hline
\gamma_{0, i} & 0 & 0 & 0 & 0  & 0  & 0 & 0\\
\gamma_{1,i} & 0 & n & -1 & -1 & -1 & 0 & 0 \\
\gamma_{2, i} & 0 & 0 & 0  & 0 & 0 & 1& 0\\
\hline
\end{array}
\]

By Proposition~\ref{proposition differences}, we can transform a local cohomology table 
$(\HH^0_\m(M)_i, \HH^1_\m(M)_i, \HH^2_\m(M)_i)$ to a point in $\mathbb M$ given by
$(\Delta^0 \HH^0_\m (M)_i, \Delta^1 \HH^1_\m(M)_i, \Delta^2 \HH^2_\m(M)_i)$.
This map is injective, and 
the extreme rays $\Lambda_2$ from Theorem~\ref{THM decomposition} 
map to $\Delta \Lambda_2$ (hence, the notation).
Thus the cone $C$ corresponds to the cone of the local cohomology tables.

The space $\mathbb M$ is naturally filtered by bounding the support of its elements:
\[
\mathbb M_{[a,b]} := \{A = \{a_{i, n}\}\in \mathbb M \mid a_{i, n} = 0 \text{ for all } i \in \{0,1, 2\}, n < a \text{ and } n > b\}.
\]
Moreover, this decomposition is compatible with the construction of $C$. Namely  $C_{[a,b]}$, the cone spanned by the rays supported in $\mathbb M_{[a,b]}$ is the intersection $\mathbb M_{[a,b]} \cap C$.

Now we will define the functionals on $\mathbb M$ that will give us the facet equations. 

\begin{definition} \label{defn hyperplanes} Let $A = \{a_{i,j}\} \in \mathbb M$. 
For $s \in \ZZ$, we set $\tau_s(A)= a_{1,s} + \sum_{i \leq s-1} a_{2,i}$, $\mu_s (A) = a_{0, s}$, 
and $\phi_s(A)= a_{2,s}$. Finally, for an integer $n \geq 0$ and $s \in \ZZ$, we set
\[
\ds \pi_{n,s}(A)= \sum_{i > s+n} a_{1,i} + (n+1)a_{1,s+n} + \sum_{i=0}^{n-1} (i+1)a_{2,s+i}.
\] 
We let $\mathcal{H}$ be the set of functionals on the space $\mathbb M$ defined by these equations.
\end{definition}

We want to show that for all $a< b$ the cone $C_{[a,b]}$ is cut by the hyperplanes defined by the functionals belonging to $\mathcal H$, thus proving that $\mathcal H$ give the facet equations of $C = \cup C_{[a, b]}$. 
By invariance under shifts, it is enough to consider $C_{[0, d]}$. For $d \geq 0$, consider the following list of functionals
\[
\ds \mathcal H_{[0,d]}= \left\{ \begin{array}{ll}
\mu_s & \text{ for } 0 \leq s \leq d, \\
\tau_s   & \text{ for } 0 \leq s < d, \\
\phi_s & \text{ for } 0 \leq s \leq d, \\
\pi_{0,s} & \text{ for } 1 \leq s \leq d, \\
\pi_{n,s} & \text{ for } 1 \leq n \leq d - 2,  1 \leq s < d - n
\end{array}\right\}.
\]
The following theorem allows us to describe the facets of the cone $C_{[0,d]}$, by identifying it with the cone defined by the list of functionals $\mathcal{H}_{[0,d]}$. 
\begin{theorem}\label{facet thm}
For $d \geq 0$, let $D_{[0,d]}$ be the cone defined by $\{A \in \mathbb M_{[0,d]} \mid H(A) \geq 0$ for all $H \in \mathcal{H}_{[0,d]}\}$. Then $D_{[0,d]} = C_{[0, d]}$.
\end{theorem}
\begin{proof}
We have $C_{[0, d]} \subseteq D_{[0, d]}$ by direct verification of positivity. 
The other inclusion will be proven by providing Algorithm~\ref{facet algo}.
\end{proof}

\begin{remark}\label{restrict H}
We may identify $\mathbb M_{[0, d- 1]}$ with a subset of $\mathbb M_{[0,d]}$ given by $\pi_{0, d} (A) = 0$, $\phi_d (A) = 0$ and $\mu_d (A) = 0$,
or, simply, $a_{0,d} = a_{1, d} = a_{2, d} = 0$.
Via this identification we have
$D_{[0,d-1]} = D_{[0,d]} \cap \mathbb M_{[0, d-1]} $.

One inclusion is clear, because $\mathcal H_{[0,d]}$ still provides non-negative functionals by restriction.
However, the nonzero restrictions that are not in $\mathcal H_{[0,d-1]}$
are now a positive linear combination of the functionals in $\mathcal H_{[0,d - 1]}$. 
Namely, for $n \geq 1$ we have a decomposition 
$\pi_{n, d - 1 - n} = (n + 1) \pi_{0, d} + n\phi_{d - 1} + \cdots + \phi_{d - 1 - n}$ as functionals on $\mathbb M_{[0, d- 1]}$.
\end{remark}

\begin{remark}
It can be checked by testing appropriate points that the list of functionals $\mathcal{H}_{[0,d]}$ minimally defines the cone $D_{[0,d]}$. In other words, removing any of the functionals would define a strictly larger cone than $D_{[0,d]}$. It then follows from Theorem~\ref{facet thm} %{\color{blue}MENTION MINIMALITY} 
that $C_{[0,d]}$ has an equal number of extremal rays and facets.
In fact, computations on Macaulay 2 suggest that the entire $f$-vector is symmetric.
This may make the reader suspect that $C_{[0,d]}$ is self-dual,  
however, the incidence matrix of $C_{[0, 4]}$ cannot be turned into a symmetric matrix by reordering rays and facets:
there is precisely one facet which contains $14$ extreme rays, $\tau_3 (x) = 0$, 
but two extreme rays that belong to $14$ facets, $E_{1, 3}$ and $E_{1,4}$. 
It is still possible, although unlikely, that the entire cone $C$ is self-dual.
\end{remark}

\subsection{Proofs}

We start with lemmas describing relations between $\Delta \Lambda_2$ and $\mathcal H$.

\begin{definition}
For $A \in \mathbb M_{[0,d]}$, we define 
$
\Supp (A) = \{H \in \mathcal H_{[0,d]} \mid H(A) \neq 0\}.
$
\end{definition}

\begin{lemma}\label{Lem Not Gamma}
For $\pi_{n,s} \in \mathcal H_{[0,d]}$ and $1 \leq k \leq d -1$ we have 
$\pi_{n, s} (\Gamma_{d-1-k}(k)) = \max (0, s - d + k)$. 
In particular, 
\[
\Supp(\Gamma_{d-1-k}(k)) = \{\phi_{d - 1 - k}; \pi_{0, s}, d + 1 - k \leq s \leq d; 
\pi_{n,s} \text{ for }1 \leq n \leq \max (0, k - 2), d + 1 - k \leq s \leq d-n-1
\}.
\] 
\end{lemma}
\begin{proof}
It is straightforward to check the functionals $\phi_s, \mu_s,$ and $\tau_s$. 
Recall that
\[
\Gamma_{d - 1 - k}(k) = 
\begin{array}{|c||c|c|c|c|c|c|}
\hline
i & d & d-1 & \cdots & d-k & d-k-1 & \cdots\\
\hline
\gamma_{0, i}  & 0 & 0 & 0  & 0  & 0 & 0\\
\gamma_{1,i}  & k & -1 & -1 & -1 & 0 & 0 \\
\gamma_{2, i} & 0 & 0  & 0 & 0 & 1& 0\\
\hline
\end{array},
\]
from which it is also clear that $\pi_{0,s} (\Gamma_{d - 1 - k}(k)) = \max (0, s - d + k)$.
Now, we consider $\pi_{n,s}$ with $n > 0$ starting with $s \leq d - 1 - k$. By the formula for $\pi_{n,s}$, we get (recall that $s + n < d$) that
\[
\pi_{n,s} (\Gamma_{d - 1 - k}(k)) = (d - k - s) - (n+1) - \sum_{i = s+n+1}^{d - 1} 1 + k
= 0.
\]
If $d - 1 - k < s$, then we only have contribution from $\gamma_{1,s}$:
\[
\pi_{n,s} (\Gamma_{d - 1 - k}(k)) = - (n+1) - \sum_{i = s+n+1}^{d - 1} 1 + k 
= s -d + k.
 \]
\end{proof}

The following relations on our equations are essential for the algorithm. 

\begin{lemma} \label{triangular identities}
For all $0 \leq i < d$ and $k > 0$ with $k+i<d$, we have
\[\phi_{i} + 2 \pi_{k, i + 1} = \pi_{k + 1, i} + \pi_{k - 1, i + 2} \text{ and }
\phi_i + 2 \pi_{0, i + 1} = \pi_{1, i} + \pi_{0, i + 2}.
\]
\end{lemma}
\begin{proof}
We first check the first equality. For a matrix $A = \{a_{i,j}\}$, the left-hand side is
\[
\ds \phi_{i}(A) + 2\pi_{k,i+1}(A) = a_{2,i} + \sum_{j = k+i+2}^d 2a_{1,j} + 2(k+1)a_{1,k+i+1} + \sum_{j=0}^{k-1} 2(j+1) a_{2,j+i+1}.
\]
The right-hand side, on the other hand, is
\begin{align*}
\ds \pi_{k + 1, i}(A) + \pi_{k - 1, i + 2} (A) & = \left[\sum_{j=k+i+2}^d a_{1,j} + (k+2)a_{1, k+i+1} + \sum_{j=0}^{k} (j+1) a_{2,j+i}\right] \\& +  \left[\sum_{j=k+i+2}^d a_{1,j} + ka_{1,k+i+1} + \sum_{j=0}^{k-2} (j+1) a_{2,j+i+2}\right] \\
& = \left[\sum_{j=k+i+2}^d a_{1,j} + (k+2)a_{1,k+i+1} + a_{2,i} + 2a_{2,i+1} + \sum_{j=2}^{k} (j+1) a_{2,j+i}\right]  \\
& + \left[\sum_{j=k+i+2}^d a_{1,j} + ka_{1,k+i+1} + \sum_{j=2}^{k} (j-1) a_{2,j+i}\right] \\
& = \sum_{j = k+i+2}^d 2a_{1,j} + 2(k+1)a_{1,k+i+1} + a_{2,i} + 2a_{2,i+1} + \sum_{j=2}^{k} 2j a_{2,j+i},
\end{align*}
and the two sides are then easily seen to agree. For the second relation, the left-hand side is
\[
\ds \phi_i(A)+2\pi_{0,i+1}(A) = a_{2,i} + \sum_{j = i+1}^d 2a_{1,j},
\]
which coincides with the right-hand side:
\[
\ds \pi_{1,i}(A) + \pi_{0,i+2}(A) = \sum_{j = i+1}^d a_{1,j} + a_{i+1} + a_{2,i} + \sum_{j=i+2}^d a_{1,j}.
\]
\end{proof}

\begin{algorithm}\label{facet algo}
Let $A = \{a_{i,j}\} \in \mathbb M_{[0,d]}$ be in the cone $D_{[0,d]}$ defined, as described above, by the set of functionals $\mathcal H_{[0, d]}$. The strategy of the following algorithm is to reduce $A$ to $0$ by subtracting a finite positive linear combination of tables from $\Delta\Lambda_2 \cap \mathbb M_{[0,d]}$. To do so, we will induct on $d \geq 1$.

\begin{itemize}
\item[{\bf Step 0:}] Replace $A$ by 
$A - a_{2, d}E_{2, d} - \sum_{i = 0}^d a_{0,i}E_{0, i}$. Set $w = d$.  Proceed to {\bf Step 1}. 

\item[{\bf Step 1:}] 
If $w = 0$ then proceed to {\bf Step 3}.
Replace $A$ with $A - a_{2, w - 1}E_{2, w - 1}$. 
Set $k = 1$ and proceed to {\bf Step 2}.

\item[{\bf Step 2:}] 
If $k = w$ then proceed to {\bf Step 3}. If $a_{1, w} = 0$, then set $w = w - 1$ and return 
to {\bf Step 1}. 
Set 
\begin{align*}
m &= \min \{\phi_{w- 1 - k} (A), \pi_{n,s} (A)/(s-w+k) \mid w + 1 - k \leq s \leq w,  0 \leq n \leq \max (0, \min (k - 2, w-s - 1))\}.
\end{align*}
Replace $A$ with $A - m\Gamma_{w - 1 - k}(k)$. Set $k = k + 1$. Repeat {\bf Step 2}.

\item[{\bf Step 3:}] 
Replace $A$ with $A - \sum_{i = 0}^d a_{1, i} E_{1, i}$. 
\end{itemize}
\end{algorithm}
\begin{proof}
Both cycles described in the algorithm are finite, so it will terminate in finitely many steps. 
We need to show that $A = 0$ at the end of the algorithm and all appearing coefficients are non-negative.
We will use induction on $d$. In the base case of  $d = 1$ we 
note that the algorithm provides us the decomposition 
\begin{align*}
A &= a_{0, 0} E_{0, 0} + a_{0, 1}E_{0, 1} + a_{1, 0} E_{1, 0} + a_{1, 1} E_{1, 1} + a_{2, 0} E_{2, 0} + a_{2, 1} E_{2, 1}\\
&= \mu_0 (A) E_{0, 0} + \mu_1(A) E_{0, 1} + \tau_0 (A) E_{1, 0} + \pi_{0,1}(A) E_{1, 1} + \phi_0(A) E_{2, 0} + \phi_1(A) E_{2, 1}.
\end{align*}
We will use this strategy in general by expressing the coefficients in terms of $H(A)$ for $H \in \mathcal H$ and 
showing that $A$ remains in the cone defined by $\mathcal H_{[0, d]}$ throughout the algorithm.

We note that, for $n = d -1$ or $d$,
$a_{2, n} = \phi_n (A) \geq 0$ and $H (A - a_{2, n}E_{2, n}) = H(A) \geq 0$ for all $\phi_n \neq H \in \mathcal H_{[0,d]}$ 
as one can easily check that $\{\phi_n\} = \Supp(E_{2, n})$. 
Similarly, for all $n$, $a_{0, n} = \mu_n (A)$ and $H(A - \mu_n (A)E_{0,n}) = H(A) \geq 0$ for all $\mu_n \neq H \in \mathcal H_{[0,d]}$.
Hence {\bf Steps 0} and {\bf 1} produce a table that is still inside $D_{[0,d]}$. 
In {\bf Step 2}, let us concentrate first on the case of $w = d$.
In this case, Lemma~\ref{Lem Not Gamma} explains that the functionals used in the definition of $m$
are exactly $\Supp(\Gamma_{d - 1 - k} (k))$. 
It follows that $H (A - m\Gamma_{d - 1 - k} (k)) = H(A) \geq 0$ for $H \notin \Supp(\Gamma_{d - 1 - k} (k) )$.
Moreover, by its definition $m \geq 0$ and $H (A - m\Gamma_{d - 1 - k} (k)) \geq 0$ for $H \in \Supp(\Gamma_{d - 1 - k} (k) )$
by Lemma~\ref{Lem Not Gamma}.

Now, we want to show that induction allows us to assume that $w = d$. To do so, we observe that we may
shrink the window $[0,d]$ after finishing the loop in {\bf Step 2}. 
\begin{claim}
For $w = d$, if repeating {\bf Step 2} does not result in $a_{1, d} = 0$ ({\it i.e., } $k = d$ is reached),
then $a_{2, 0} = \cdots = a_{2, d} = 0$.
\end{claim}
\begin{proof}
Let $A$ be the matrix at the beginning of {\bf Step 2} and 
$A' = A - m \Gamma_{d - 1 - k} (k)$, the result of the step.
By induction on $k$ we show that either $a'_{1, d} = 0$
or in $a'_{2, d - k - 1} = \cdots = a'_{2, d} = 0$. 
At $k = 1$, we either have $m = \phi_{d - 2} (A)$ or $m = \pi_{d} (A)$.
In the second case $a'_{1, d} = \pi_{0, d} (A') =  0$ and the claim follows, or 
$a'_{2, d - 2} = \phi_{d-2} (A') = 0$.

By the induction hypothesis, we may assume that $a_{2, d - k - 2} = \cdots = a_{2, d} = 0$.
If $m = \phi_{d - 1 - k} (A)$ then  $a'_{2, d - k - 1} = \cdots = a'_{2, d} = 0$ and we are done.
Otherwise,  $m = \pi_{n,s} (A)/\pi_{n,s} (\Gamma_{d - 1 - k} (k))$, so $\pi_{n,s} (A') = 0$. 
If $n > 0$, then $d + 1 -k \leq s \leq d - n - 1$, so we may use Lemma \ref{triangular identities} to show that $\pi_{0, s + n} (A') = 0$.
If $n = 0$, then necessarily $d + 1 - k \leq s \leq d$, so we may use again Lemma \ref{triangular identities} and the fact that $A' \in D_{[0,d]}$ to show that $\pi_{0, d} (A') = 0$ as well.

\end{proof}
If $a_{2, 0} = \cdots = a_{2, d} = 0$, then the equations $T_{i}$ and $P_{0, d}$ 
show that $a_{1, i} \geq 0$ for $0 \leq i \leq d$. Hence when 
we are moved to {\bf Step 3}, we subtract a positive linear combination of $E_{1, i}$ and the resulting table is $0$.
Otherwise, when we leave {\bf Step 2}, $a_{0, d} = a_{1,d} = a_{2, d} = 0$ and we may now consider $A$ as a table 
in $D_{[0, d- 1]}$ by Remark~\ref{restrict H}. This concludes the induction step. 
\end{proof}

\section*{Acknowledgments} We thank Mats Boij, Giulio Caviglia, Daniel Erman, Alessio Sammartano, and Greg Smith for several helpful conversations. We thank the anonymous referee for very helpful comments, and for suggesting a shorter proof of Theorem \ref{THM decomposition}.

\bibliographystyle{alpha}
\bibliography{References}
\end{document}